\documentclass[11pt,a4paper]{polyshape}

\usepackage{amsthm}
\usepackage{enumitem}
\usepackage{natbib}
\usepackage{subcaption}
\usepackage[colorinlistoftodos]{todonotes}

\newcommand{\CMP}{\mathbf{CMP}}
\newcommand{\E}{\mathcal{E}}
\newcommand{\Hom}{\mathrm{Hom}}
\newcommand{\id}{\mathrm{id}}
\newcommand{\Id}{\mathrm{Id}}
\newcommand{\Int}{\mathrm{Int}}
\newcommand{\K}{\mathcal{K}}
\newcommand{\M}{\mathcal{M}}
\newcommand{\Ob}{\mathrm{Ob}}
\newcommand{\PC}{\mathcal{P}}
\newcommand{\T}{\mathcal{T}}
\newcommand{\ThetaC}{\mathcal{O}}

\newtheorem{corollary}{Corollary}
\newtheorem{definition}{Definition}
\newtheorem{example}{Example}
\newtheorem{lemma}{Lemma}
\newtheorem{proposition}{Proposition}
\newtheorem{theorem}{Theorem}

\setlist[enumerate]{leftmargin=*, topsep=0pt, itemsep=0pt}

\uselogo

\title{Colored Markov polycategories and diagrammatic differentiation}

\author[1,2]{Theodore Papamarkou}
\affil[1]{PolyShape, Greece}
\affil[2]{National Technical University of Athens, Greece}

\correspondingauthor{theodore@polyshape.com}
\reportnumber{}

\date{}

\begin{document}

\maketitle

\begin{abstract}
Many stochastic systems are built by wiring typed components together, but the wiring is often neither purely sequential nor type-homogeneous. This paper develops categorical semantics for such systems using ordered polycategories whose morphisms are Markov kernels. The basic operation is kernel slotwise composition, which connects one output slot of a many-output kernel to one input slot of another and marginalizes the internal wire. We prove its structural laws by assigning trace semantics to finite acyclic diagrams. We then introduce colored Markov polycategories, where objects and kernels carry colors and typed connections are realized by interface kernels satisfying identity and composition laws. This gives a colored kernel slotwise composition and trace semantics for typed stochastic diagrams. To describe systems whose structure changes, we co-index colored Markov polycategories and parameter spaces over an indexing category. Finally, for finite acyclic parameterized diagrams, we prove a diagrammatic differentiation result. The derivative of an expected scalar objective is obtained from local reverse-mode contributions at the parameterized vertices, with stochastic and deterministic kernels handled through local gradient representing functions. The construction gives a typed, compositional language for finite acyclic stochastic systems and their parameter sensitivities.
\end{abstract}

\section{Introduction}
\label{sec:introduction}

Many stochastic systems are assembled from typed components whose interactions are not purely sequential. A component may have several inputs and several outputs, and only some of its output wires may be connected to later components. The composition of Markov kernels typically captures sequential stochastic computation, without recording the slotwise wiring of many-input, many-output components or the typed compatibility constraints between the wires being connected. This paper develops compositional semantics for such systems using ordered polycategories whose morphisms are interpreted as Markov kernels.

The central structure introduced here is the colored Markov polycategory. Its objects carry standard Borel spaces and object colors, its morphisms are Markov kernels with morphism colors, and its typed wire connections are mediated by interface kernels. We first define uncolored kernel slotwise composition and prove its structural laws by giving trace semantics for finite acyclic diagrams. We then add colors and interface systems, obtaining colored kernel slotwise composition and corresponding trace semantics for typed diagrams. We also define co-indexed colored Markov polycategories, in which colored Markov polycategories and parameter spaces vary functorially over an indexing category. Finally, we prove a bounded diagrammatic differentiation result for finite acyclic parameterized diagrams. The result identifies the gradient of an expected scalar objective with the expectation of local reverse-mode contributions at the parameterized vertices.

The scope of the paper is restricted to the hypotheses under which the trace and differentiation results are proved. All diagrams are finite and acyclic, all measurable spaces are standard Borel, and all interface color morphisms are fixed as part of the diagram. The differentiation applies to a parameterized diagram in a colored Markov polycategory, while the co-indexed structure transports parameters between states by differentiable pushforwards.

\section{Background and notation}
\label{sec:background}

This section presents the notation used throughout the paper. It recalls the background needed for the later constructions, namely standard Borel spaces and Markov kernels, ordered polycategories, string diagrams, colors, and co-indexed structures.

\subsection{Standard Borel spaces and Markov kernels}
\label{subsec:sbs_markov_kernels}

The probabilistic semantics in this paper is based on Markov kernels between standard Borel spaces. A measurable space $(S,\Sigma_S)$ is standard Borel if it is isomorphic, as a measurable space, to a Borel subset of a Polish space equipped with its Borel $\sigma$-algebra. Standard Borel spaces are closed under finite products.

Each object $X$ considered in this paper is assigned an underlying standard Borel space, denoted by $(\underline{X},\Sigma_{\underline{X}})$. A value of type $X$ is denoted by a lower-case letter $x\in \underline{X}$. If $\mathbf{X}=(X_1,\ldots,X_m)$ is a finite ordered tuple of objects, then its underlying measurable space is the product standard Borel space $(\underline{\mathbf{X}}, \Sigma_{\underline{\mathbf{X}}}) = (\prod_{r=1}^m \underline{X}_r, \bigotimes_{r=1}^m \Sigma_{\underline{X}_r})$. A value of type $\mathbf{X}$ is written $\mathbf{x}=(x_1,\ldots,x_m)$. When $m=0$, the product is interpreted as the one-point standard Borel space.

A Markov kernel from an object $X$ to an object $Y$ is a map $k:(\Sigma_{\underline{Y}},\underline{X})\to [0,1]$ such that $k(\cdot\mid x)$ is a probability measure on $(\underline{Y},\Sigma_{\underline{Y}})$ for each $x\in \underline{X}$, and $x\mapsto k(E\mid x)$ is $\Sigma_{\underline{X}}$-measurable for each $E\in \Sigma_{\underline{Y}}$. We write $k(dy\mid x)$ for the probability measure $k(\cdot\mid x)$.

We also write $k:X\to Y$ as a shorthand for the kernel $k:(\Sigma_{\underline{Y}},\underline{X})\to [0,1]$. For tuple-valued kernels, the notation $k:\mathbf{X}\to \mathbf{Y}$ similarly abbreviates the kernel signature $k:(\Sigma_{\underline{\mathbf{Y}}},\underline{\mathbf{X}})\to [0,1]$. In this paper, polycategorical morphisms are interpreted as Markov kernels, so we identify a morphism with its kernel interpretation and use the same signature notation for both.

If $k:X\to Y$ and $l:Y\to Z$ are Markov kernels, their composite $l\circ k:X\to Z$ is the Markov kernel defined by
\[
(l\circ k)(E\mid x)
=
\int_{\underline{Y}} l(E\mid y)\,k(dy\mid x)
\]
for every $E\in \Sigma_{\underline{Z}}$ and $x\in \underline{X}$. This is the Chapman-Kolmogorov composition of kernels \citep{kallenberg2021}. The identity kernel on $X$ is the kernel $\id_X:X\to X$ given by $\id_X(E\mid x)=\delta_x(E)$, where $\delta_x$ is the Dirac probability measure at $x$.

A measurable map $f:\underline{X}\to \underline{Y}$ determines a deterministic Markov kernel $\delta_f:X\to Y$ by $\delta_f(E\mid x)=\delta_{f(x)}(E)$. Thus, deterministic maps are treated as special cases of Markov kernels.

For finite ordered tuples $\mathbf{X}=(X_1,\ldots,X_m)$ and $\mathbf{Y}=(Y_1,\ldots,Y_n)$, a Markov kernel from $\mathbf{X}$ to $\mathbf{Y}$ means a Markov kernel from $(\underline{\mathbf{X}},\Sigma_{\underline{\mathbf{X}}})$ to $(\underline{\mathbf{Y}},\Sigma_{\underline{\mathbf{Y}}})$. These tuple-valued kernels are the probabilistic interpretation of the many-input, many-output morphisms used in this paper.

The categorical perspective on Markov kernels is related to the theory of Markov categories \citep{fritz2020}. The present paper uses Markov kernels as the semantic interpretation of polycategorical morphisms and introduces slotwise composition separately.

\subsection{Ordered polycategories}
\label{subsec:ordered_polycategories}

This paper uses polycategories rather than multicategories. A multicategory has many-input, single-output morphisms, while a polycategory has many-input, many-output morphisms. Since the kernels studied here have signatures of the form $k:\mathbf{A}\to \mathbf{B}$, with both $\mathbf{A}$ and $\mathbf{B}$ finite ordered tuples, the many-output structure is required.

Polycategories were introduced by \citet{szabo1975}. There are several variants in the literature. \citet{garner2008} gives a modern abstract account of symmetric polycategories, while non-symmetric and planar variants occur in related work on weakly distributive structures and monadic approaches to polycategories \citep{cockettseely1997,koslowski2005}. The word ``planar'' can carry additional no-crossing restrictions beyond the convention needed here. For this reason, we use the phrase ``ordered polycategory'' to refer only to the strict ordered-list convention.

Thus, input and output profiles are finite ordered tuples of objects. Equality of profiles means equality of ordered tuples, as opposed to equality up to permutation. Let $\PC$ be such a polycategory. Its objects form a collection $\Ob(\PC)$. For finite ordered tuples of objects $\mathbf{A}=(A_1,\ldots,A_m)$ and $\mathbf{B}=(B_1,\ldots,B_n)$, the set of morphisms from $\mathbf{A}$ to $\mathbf{B}$ is denoted by $\Hom_{\PC}(\mathbf{A};\mathbf{B})$. A morphism $k\in \Hom_{\PC}(\mathbf{A};\mathbf{B})$ is written $k:\mathbf{A}\to \mathbf{B}$. Empty input or output tuples are allowed.

For every object $X\in \Ob(\PC)$, there is an identity morphism $\id_X:(X)\to (X)$. Composition is defined slotwise. Let $k:\mathbf{A}\to \mathbf{B}$ and $l:\mathbf{C}\to \mathbf{D}$, where $\mathbf{A}=(A_1,\ldots,A_m)$, $\mathbf{B}=(B_1,\ldots,B_n)$, $\mathbf{C}=(C_1,\ldots,C_q)$, and $\mathbf{D}=(D_1,\ldots,D_r)$. If the $i$-th output object of $k$ equals the $j$-th input object of $l$, that is, if $B_i=C_j$, then their slotwise composite is written $l\circ_{(i,j)}k$. Its input profile is
\[
(C_1,\ldots,C_{j-1},A_1,\ldots,A_m,C_{j+1},\ldots,C_q),
\]
and its output profile is
\[
(B_1,\ldots,B_{i-1},D_1,\ldots,D_r,B_{i+1},\ldots,B_n).
\]

The identity morphisms are units for slotwise composition. Iterated slotwise compositions satisfy the associativity and interchange axioms for polycategories. Equivalently, any two iterated composites with the same underlying finite acyclic wiring diagram and the same external ordered input and output profiles are equal.

The insertion convention in the displayed profiles is used throughout the paper, retaining the order of wires and avoiding permutations. This allows the structural laws for kernel slotwise composition to be stated as equalities of kernels with identical ordered signatures.

\subsection{String diagrams for polycategorical composition}
\label{subsec:string_diagrams_polycats}

String diagrams are used as a graphical notation for slotwise composition. A morphism $k:\mathbf{A}\to \mathbf{B}$ is drawn as a box with input wires labelled by the ordered tuple $\mathbf{A}$ and output wires labelled by the ordered tuple $\mathbf{B}$. For $k:\mathbf{A}\to \mathbf{B}$ and $l:\mathbf{C}\to \mathbf{D}$, the slotwise composite $l\circ_{(i,j)}k$ is drawn by connecting the $i$-th output wire of $k$ to the $j$-th input wire of $l$.

The diagrams in this paper follow the ordered convention of Subsection~\ref{subsec:ordered_polycategories}. Thus a diagram is read together with the ordered external input and output profiles and with the indicated slotwise connections.

Only finite acyclic string diagrams are considered. A finite acyclic diagram is built from finitely many boxes by finitely many slotwise connections, with no directed cycle of wires. An unconnected input slot is an external input wire, and an unconnected output slot is an external output wire. The ordered tuple of all external input wires is the input profile of the diagram, and the ordered tuple of all external output wires is the output profile of the diagram.

The string diagrams are also linear in the following sense. Each input slot is connected to at most one output slot, and each output slot is connected to at most one input slot. There is no implicit copying, merging, or discarding of wires. Such operations can be represented only if they are supplied as morphisms. This convention is required because the kernel semantics developed here treats each wire as carrying one value of its assigned object type.

For a finite acyclic diagram, its boxes are interpreted as Markov kernels. Its internal wires are integrated out, and the external input and output wires determine the Markov kernel represented by the whole diagram. This interpretation is developed first for uncolored diagrams in Section~\ref{sec:markov_polycategories_ksc} and then for colored diagrams in Section~\ref{sec:cksc_diagram_semantics}.

\subsection{Colors and typed interfaces}
\label{subsec:colors_typed_interfaces}

The word color is used in the standard categorical sense of a type or sort. In the present paper, colors are used to record which objects and morphisms may be composed.

There are two kinds of colors. Object colors classify the objects carried by wires. They are organized by an object-color category $\K$. If $X$ is an object, its object color is written $\chi(X)\in \Ob(\K)$. The object-coloring notation is extended componentwise to ordered profiles. If $\mathbf{A}=(A_1,\ldots,A_m)$, then $\chi(\mathbf{A})=(\chi(A_1),\ldots,\chi(A_m))$. Morphism colors classify the roles or types of boxes. They are organized by a morphism-color polycategory, denoted by $\M$. If $k$ is a morphism, its morphism color is written $\psi(k)\in \Hom_{\M}(\chi(\mathbf{A});\chi(\mathbf{B}))$.

Object colors control whether an output wire can be connected to an input wire. In the uncolored setting of Subsection~\ref{subsec:ordered_polycategories}, a slotwise composite $l\circ_{(i,j)}k$ requires literal equality $B_i=C_j$. In the colored setting, this condition is relaxed. The objects $B_i$ and $C_j$ may be different, provided that their colors are compatible. This compatibility is specified by a morphism $f:\chi(B_i)\to \chi(C_j)$ in $\K$.

A color-compatible connection may require a transformation between the corresponding measurable spaces. This transformation is represented by an interface kernel. If $f:\chi(B)\to \chi(C)$ is a compatible object-color morphism, the associated interface kernel is written $\kappa_{f;B,C}:B\to C$. The notation includes the objects $B$ and $C$, and the color morphism $f$, because the same color-level relation may apply to different underlying measurable spaces.

The definition of interface kernels and their identity and composition laws are given in Section~\ref{sec:colored_markov_polycategories}. The convention is that colors specify typed compatibility, while interface kernels specify the probabilistic or deterministic data transformation needed to realize a typed connection.

Morphism colors are used separately from object colors. They classify boxes by their typed input and output profiles. Identity morphism colors in $\M$ are written $\id_c:(c)\to(c)$ for object colors $c$.

\subsection{Co-indexed structures and parameter spaces}
\label{subsec:coindexed_structures_parameters}

The dynamic part of the paper uses co-indexed structures. The indexing category is denoted by $\T$. Its objects are written $t,t',t''$, and its morphisms are written $\alpha:t \to t'$. In applications, an object $t$ may represent a time, a system state, or a structural configuration, while a morphism $\alpha:t \to t'$ represents a transition from one such state to another.

A co-indexed family of structures over $\T$ is treated here as a covariant functor out of $\T$. Thus, if $\E$ is a category of mathematical structures, a co-indexed structure is a functor $F:\T \to \E$. For each $t\in \Ob(\T)$, the object $F(t)$ is the structure assigned to state $t$. For each morphism $\alpha:t\to t'$ in $\T$, the morphism $F(\alpha):F(t) \to F(t')$ is the corresponding pushforward along $\alpha$. This convention for co-indexing \citep{jacobs1999,johnstone2002} is dual to the contravariant convention used for indexed categories.

In the present paper, the main state-level structures will be colored Markov polycategories. Before introducing them formally, we introduce the notation for pushforwards. If $\PC_t$ is the state structure at $t$, then the pushforward along $\alpha:t\to t'$ is written $\alpha_{!,s}:\PC_t \to \PC_{t'}$. The subscript $s$ indicates that this is a state pushforward. The exclamation mark is used to emphasize covariance.

The parameter spaces are treated in parallel. For each state $t$, there is a finite-dimensional real vector space $\Theta_t$ of parameters. These spaces form a category $\ThetaC$ whose objects are the parameter spaces and whose morphisms are differentiable maps between them. A parameter pushforward along $\alpha:t\to t'$ is written $\alpha_{!,\theta}:\Theta_t \to \Theta_{t'}$, and specifies how parameters are transported when the state changes.

The formal definition of a co-indexed colored Markov polycategory uses two functors, $F_{\PC}:\T \to \CMP,~F_\theta:\T \to \ThetaC$, where $\CMP$ denotes the category of colored Markov polycategories. The first functor assigns a colored Markov polycategory to each state. The second functor assigns a parameter space to each state. In this paper, these functors are taken to be strict. Hence, identities and composites in $\T$ are preserved, that is, $F_{\PC}(\id_t)=\Id_{\PC_t},~F_{\PC}(\beta\circ \alpha)=F_{\PC}(\beta)\circ F_{\PC}(\alpha)$, and similarly $F_\theta(\id_t)=\id_{\Theta_t},~F_\theta(\beta\circ \alpha)=F_\theta(\beta)\circ F_\theta(\alpha)$. This strict convention keeps the later learning and diagrammatic semantics statements free of coherence isomorphisms.

\section{Markov polycategories and kernel slotwise composition}
\label{sec:markov_polycategories_ksc}

This section develops the uncolored calculus of kernel slotwise composition. The purpose is to isolate the measure-theoretic content of slotwise composition before adding colors and typed interfaces. Colors are introduced in Section~\ref{sec:colored_markov_polycategories}. Proofs not included in the main text of this section are given in Appendix~\ref{app:ksc_proofs}.

\subsection{Markov polycategories}
\label{subsec:markov_polycategories}

We first define the Markov polycategory datum used in this section. Definition~\ref{def:markov_polycategory_data} records the objects, their measurable spaces, and the kernel-valued morphisms. The slotwise composition operation is introduced in Subsection~\ref{subsec:ksc}.

\begin{definition}[Markov polycategory datum]
\label{def:markov_polycategory_data}
A Markov polycategory datum $\PC$ consists of the following data.
\begin{enumerate}
\item $\PC$ has a collection $\Ob(\PC)$ of objects. Each object $X\in \Ob(\PC)$ is assigned a standard Borel space $(\underline{X},\Sigma_{\underline{X}})$.
\item For every pair of finite ordered tuples $\mathbf{A}=(A_1,\ldots,A_m)$ and $\mathbf{B}=(B_1,\ldots,B_n)$ of objects of $\PC$, there is a set $\Hom_{\PC}(\mathbf{A};\mathbf{B})$ of morphisms from $\mathbf{A}$ to $\mathbf{B}$.
\item Every morphism $k\in \Hom_{\PC}(\mathbf{A};\mathbf{B})$ is a Markov kernel $k:(\Sigma_{\underline{\mathbf{B}}},\underline{\mathbf{A}})\to [0,1]$. Thus, for every $\mathbf{a}\in \underline{\mathbf{A}}$, the expression $k(d\mathbf{b}\mid \mathbf{a})$ denotes a probability measure on $(\underline{\mathbf{B}},\Sigma_{\underline{\mathbf{B}}})$.
\item For every object $X\in \Ob(\PC)$, the identity morphism $\id_X\in \Hom_{\PC}((X);(X))$ is the identity Markov kernel $\id_X(E\mid x)=\delta_x(E)$ for every $E\in \Sigma_{\underline{X}}$ and $x\in \underline{X}$.
\end{enumerate}
\end{definition}

A Markov polycategory datum becomes a Markov polycategory once it is equipped with a slotwise composition operation satisfying the unit and associativity axioms of Subsection~\ref{subsec:ordered_polycategories}. Subsection~\ref{subsec:ksc} defines such an operation by an integral formula. The resulting operation is called kernel slotwise composition.

The word ``datum'' in Definition~\ref{def:markov_polycategory_data} is used to avoid circularity. At this point, the kernel-valued morphisms and identity kernels have been specified, but the composition of two many-input, many-output kernels has not yet been defined. The next subsection gives this composition.

\subsection{Kernel slotwise composition}
\label{subsec:ksc}

Let $\PC$ be a Markov polycategory datum. Let $k:\mathbf{A}\to \mathbf{B}$ and $l:\mathbf{C}\to \mathbf{D}$ be morphisms of $\PC$, where
$\mathbf{A}=(A_1,\ldots,A_m),~\mathbf{B}=(B_1,\ldots,B_n),~\mathbf{C}=(C_1,\ldots,C_q),~\mathbf{D}=(D_1,\ldots,D_r)$. Let $i\in \{1,\ldots,n\}$ and $j\in \{1,\ldots,q\}$. In the uncolored setting, the slotwise composite along $(i,j)$ is defined only when $B_i=C_j$ as objects of $\PC$.

We write $\mathbf{C}_{-j}$ for the ordered tuple obtained from $\mathbf{C}$ by deleting $C_j$, and $\mathbf{B}_{-i}$ for the ordered tuple obtained from $\mathbf{B}$ by deleting $B_i$. The input and output profiles of the composite are
\[
\boldsymbol{\Gamma}
=
(C_1,\ldots,C_{j-1},A_1,\ldots,A_m,C_{j+1},\ldots,C_q)
\]
and
\[
\boldsymbol{\Delta}
=
(B_1,\ldots,B_{i-1},D_1,\ldots,D_r,B_{i+1},\ldots,B_n).
\]
For $\mathbf{a}\in \underline{\mathbf{A}}$ and $\mathbf{c}_{-j}\in \underline{\mathbf{C}_{-j}}$, let
\[
V_j(\mathbf{a},\mathbf{c}_{-j})
=
(c_1,\ldots,c_{j-1},a_1,\ldots,a_m,c_{j+1},\ldots,c_q)
\in \underline{\boldsymbol{\Gamma}}.
\]
For $\mathbf{b}\in \underline{\mathbf{B}}$, let $\pi_i(\mathbf{b})=b_i$ and $\pi_{-i}(\mathbf{b})=(b_1,\ldots,b_{i-1},b_{i+1},\ldots,b_n)$. For $z\in \underline{B_i}=\underline{C_j}$, define
\[
T_j(\mathbf{c}_{-j},z)
=
(c_1,\ldots,c_{j-1},z,c_{j+1},\ldots,c_q)
\in \underline{\mathbf{C}}.
\]
For $\mathbf{b}_{-i}\in \underline{\mathbf{B}_{-i}}$ and $\mathbf{d}\in \underline{\mathbf{D}}$, define
\[
U_i(\mathbf{b}_{-i},\mathbf{d})
=
(b_1,\ldots,b_{i-1},d_1,\ldots,d_r,b_{i+1},\ldots,b_n)
\in \underline{\boldsymbol{\Delta}}.
\]

\begin{definition}[Kernel slotwise composition]
\label{def:ksc}
The kernel slotwise composition (KSC) of $l$ with $k$ along the slot pair $(i,j)$ is the Markov kernel $l\circ_{(i,j)}k:\boldsymbol{\Gamma}\to \boldsymbol{\Delta}$ defined as follows. For every $E\in \Sigma_{\underline{\boldsymbol{\Delta}}}$ and every input $V_j(\mathbf{a},\mathbf{c}_{-j})\in \underline{\boldsymbol{\Gamma}}$,
\[
(l\circ_{(i,j)}k)(E\mid V_j(\mathbf{a},\mathbf{c}_{-j}))\\
=
\int_{\underline{\mathbf{B}}}
\left(
\int_{\underline{\mathbf{D}}}
\mathbf{1}_E\bigl(U_i(\pi_{-i}(\mathbf{b}),\mathbf{d})\bigr)
\,l(d\mathbf{d}\mid T_j(\mathbf{c}_{-j},\pi_i(\mathbf{b})))
\right)
k(d\mathbf{b}\mid \mathbf{a}).
\]
\end{definition}

The formula first samples the output tuple $\mathbf{b}$ of $k$ from the input $\mathbf{a}$. Its $i$-th component $b_i$ is then used as the $j$-th input of $l$, together with the remaining external inputs $\mathbf{c}_{-j}$. The outputs of the composite are the outputs $\mathbf{d}$ of $l$ together with the unused output components $\pi_{-i}(\mathbf{b})$ of $k$, in the ordered profile $\boldsymbol{\Delta}$.

\begin{example}[Kernel slotwise composition of Gaussian kernels]
\label{ex:ksc_gaussian}
Let $\PC$ be a Markov polycategory datum in which the objects $A,B_1,B_2,C_1,C_2,D$ all have underlying standard Borel space $(\mathbb{R},\mathcal{B}(\mathbb{R}))$. Assume that $B_1=C_1$ as objects of $\PC$.

Let $k:(A)\to (B_1,B_2)$ be the Markov kernel defined by
\[
k(d(b_1,b_2)\mid a)
=
\mathcal{N}(db_1;a,\sigma_1^2)\,\delta_a(db_2),
\]
where $\mathcal{N}(db_1;a,\sigma_1^2)$ denotes the Gaussian probability measure on $\mathbb{R}$ with mean $a$ and variance $\sigma_1^2$. Thus, conditional on $a$, the first output is Gaussian with mean $a$, while the second output is deterministically equal to $a$.

Let $l:(C_1,C_2)\to (D)$ be the Markov kernel defined by
\[
l(dd\mid (c_1,c_2))
=
\mathcal{N}(dd;c_1+c_2,\sigma_2^2).
\]
The slotwise composite $h=l\circ_{(1,1)}k$ connects the first output of $k$ to the first input of $l$. Its input profile is $(A,C_2)$ and its output profile is $(D,B_2)$.

For $E\in \mathcal{B}(\mathbb{R}^2)$ and input $(a,c_2)\in \mathbb{R}^2$, Definition~\ref{def:ksc} gives
\[
h(E\mid (a,c_2))
=
\int_{\mathbb{R}^2}
\left(
\int_{\mathbb{R}}
\mathbf{1}_E(d,b_2)\,
\mathcal{N}(dd;b_1+c_2,\sigma_2^2)
\right)
\mathcal{N}(db_1;a,\sigma_1^2)\,\delta_a(db_2).
\]
Since $b_1\mid a$ is Gaussian with mean $a$ and variance $\sigma_1^2$, and $d\mid b_1,c_2$ is Gaussian with mean $b_1+c_2$ and variance $\sigma_2^2$, the marginal conditional law of $d$ given $(a,c_2)$ is Gaussian with mean $a+c_2$ and variance $\sigma_1^2+\sigma_2^2$. The second output remains deterministically equal to $a$. Hence the composite kernel is
\[
h(d(d,b_2)\mid (a,c_2))
=
\mathcal{N}(dd;a+c_2,\sigma_1^2+\sigma_2^2)\,\delta_a(db_2).
\]

This example illustrates the two roles of kernel slotwise composition. The connected component $b_1$ is integrated out, while the unconnected output $b_2$ of $k$ passes through as an output of the composite.
\end{example}

\subsection{Trace semantics for finite acyclic KSC diagrams}
\label{subsec:trace_semantics_ksc}

Definition~\ref{def:ksc} gives a binary slotwise composition operation. We now define the semantics of a finite acyclic diagram of such compositions. The point of this formulation is that the composite kernel of a diagram is obtained by integrating out its internal wires.

\begin{definition}[Finite acyclic KSC diagram]
\label{def:finite_acyclic_ksc_diagram}
Let $\PC$ be a Markov polycategory datum. A finite acyclic KSC diagram $D$ in $\PC$ consists of the following data.
\begin{enumerate}
\item $D$ has a finite set $V(D)$ of vertices. Each vertex $u\in V(D)$ is labelled by a morphism $k_u:\mathbf{A}_u\to \mathbf{B}_u$ of $\PC$, where $\mathbf{A}_u=(A_{u,1},\ldots,A_{u,m_u})$ and $\mathbf{B}_u=(B_{u,1},\ldots,B_{u,n_u})$ are finite ordered tuples of objects.
\item $D$ has a finite set $W_{\mathrm{int}}(D)$ of internal wires. Each internal wire is a pair $e=((u,p),(v,q))$, where $u,v\in V(D)$, $1\leq p\leq n_u$, and $1\leq q\leq m_v$. It connects the $p$-th output slot of $u$ to the $q$-th input slot of $v$. It is required that $B_{u,p}=A_{v,q}$.
\item Each input slot is the target of at most one internal wire, and each output slot is the source of at most one internal wire.
\item The directed graph on $V(D)$ with an edge $u\to v$ for every internal wire $((u,p),(v,q))$ is acyclic.
\end{enumerate}

The external input slots of $D$ are the input slots $(v,q)$ that are not targets of internal wires. The external output slots of $D$ are the output slots $(u,p)$ that are not sources of internal wires.

The diagram includes an ordering of all its external input slots and all its external output slots. If the ordered list of external input slots is $((v_1,q_1),\ldots,(v_a,q_a))$, then the external input profile of $D$ is $\mathbf{I}_D=(A_{v_1,q_1},\ldots,A_{v_a,q_a})$. If the ordered list of external output slots is $((u_1,p_1),\ldots,(u_b,p_b))$, then the external output profile of $D$ is $\mathbf{O}_D=(B_{u_1,p_1},\ldots,B_{u_b,p_b})$.
\end{definition}

The third clause of Definition~\ref{def:finite_acyclic_ksc_diagram} is the linearity condition on wires. It excludes implicit copying and merging. Since every unconnected output slot is an external output slot, the diagram also has no implicit discarding.

The specified orderings of external input and output slots are part of the data because a finite acyclic diagram does not need to have a preferred left-to-right order of external wires. Once the ordered external profiles are fixed, the trace construction of Definition~\ref{def:trace_measure_ksc} assigns a Markov kernel $\mathbf{I}_D \to \mathbf{O}_D$ to the diagram.

Definition~\ref{def:finite_acyclic_ksc_diagram} gives a combinatorial presentation of a finite acyclic string diagram. Its vertices are boxes labelled by morphisms of $\PC$, and its internal wires record slotwise connections between output and input slots. The trace kernel of Definition~\ref{def:trace_measure_ksc} is the Markov kernel denoted by the diagram.

\begin{definition}[Trace measure of a finite acyclic KSC diagram]
\label{def:trace_measure_ksc}
Let $D$ be a finite acyclic KSC diagram in a Markov polycategory datum $\PC$. Let $u_1,\ldots,u_N$ be a topological ordering of the vertices $V(D)$ of $D$. For each $\ell\in \{1,\ldots,N\}$, write $k_{u_\ell}:\mathbf{A}_{u_\ell}\to \mathbf{B}_{u_\ell}$ for the kernel labelling $u_\ell$.

Let $\mathbf{x}\in \underline{\mathbf{I}_D}$ be a value of the external input profile of $D$. Since the ordering of $V(D)$ is topological, each input slot of $u_\ell$ is filled either by an external input value from $\mathbf{x}$ or by an output value of some earlier vertex $u_s$ with $s<\ell$. Hence, there is a measurable input map
\[
R_\ell:
\underline{\mathbf{I}_D}
\times
\underline{\mathbf{B}_{u_1}}
\times\cdots\times
\underline{\mathbf{B}_{u_{\ell-1}}}
\to
\underline{\mathbf{A}_{u_\ell}}
\]
that returns the ordered input tuple of $u_\ell$.

The trace measure of $D$ with respect to the topological ordering of $V(D)$ is the probability measure $\mu_D^{u_1,\ldots,u_N}(\,\cdot\mid \mathbf{x})$ on $\prod_{\ell=1}^N \underline{\mathbf{B}_{u_\ell}}$ defined by the iterated kernel product
\[
\mu_D^{u_1,\ldots,u_N}
(d\mathbf{y}_1,\ldots,d\mathbf{y}_N\mid \mathbf{x})
=
\prod_{\ell=1}^N
k_{u_\ell}
\left(
d\mathbf{y}_\ell
\mid
R_\ell(\mathbf{x},\mathbf{y}_1,\ldots,\mathbf{y}_{\ell-1})
\right),
\]
where the product notation denotes the corresponding iterated integral in the order $u_1,\ldots,u_N$.

Let
\[
P_{\mathrm{out}}:
\prod_{\ell=1}^N \underline{\mathbf{B}_{u_\ell}}
\to
\underline{\mathbf{O}_D}
\]
be the measurable map which selects the output-slot values listed in the external output profile of $D$, in the order specified by that profile. The trace kernel of $D$ with respect to the topological ordering of $V(D)$ is the Markov kernel $K_D^{u_1,\ldots,u_N}:\mathbf{I}_D\to \mathbf{O}_D$ defined by
\[
K_D^{u_1,\ldots,u_N}(E\mid \mathbf{x})
=
\int
\mathbf{1}_E
\left(P_{\mathrm{out}}(\mathbf{y}_1,\ldots,\mathbf{y}_N)\right)
\,
\mu_D^{u_1,\ldots,u_N}
(d\mathbf{y}_1,\ldots,d\mathbf{y}_N\mid \mathbf{x})
\]
for every $E\in \Sigma_{\underline{\mathbf{O}_D}}$ and $\mathbf{x}\in \underline{\mathbf{I}_D}$.
\end{definition}

The measure $\mu_D^{u_1,\ldots,u_N}(\,\cdot\mid \mathbf{x})$ samples all vertex outputs of the diagram, including those carried by internal wires. The trace kernel $K_D^{u_1,\ldots,u_N}$ then forgets the internal output values and keeps only the values on the external output wires.

\begin{theorem}[Trace semantics for finite acyclic KSC diagrams]
\label{thm:trace_semantics_ksc}
Let $D$ be a finite acyclic KSC diagram in a Markov polycategory datum $\PC$. For every topological ordering $u_1,\ldots,u_N$ of $V(D)$, the trace kernel $K_D^{u_1,\ldots,u_N}:\mathbf{I}_D\to \mathbf{O}_D$ of Definition~\ref{def:trace_measure_ksc} is a Markov kernel. Moreover, if $u_1,\ldots,u_N$ and $u'_1,\ldots,u'_N$ are two topological orderings of $V(D)$, then the corresponding trace kernels are equal: $K_D^{u_1,\ldots,u_N} = K_D^{u'_1,\ldots,u'_N}$. Hence, a finite acyclic KSC diagram $D$ determines a unique Markov kernel $K_D:\mathbf{I}_D\to \mathbf{O}_D$.
\end{theorem}

\begin{example}[Why slot indices are part of the data]
\label{ex:slot_indices_part_of_data}
Let $X$ be an object with underlying standard Borel space $(\mathbb{R},\mathcal{B}(\mathbb{R}))$. Consider two nullary morphisms $k_1:()\to (X),~k_2:()\to (X)$, interpreted as Markov kernels $k_1:(\Sigma_{\underline{X}},*)\to [0,1],~k_2:(\Sigma_{\underline{X}},*)\to [0,1]$, where $*$ is the one-point space. Let $k_1(dx\mid *)=\delta_1(dx),~k_2(dx\mid *)=\delta_2(dx)$. Let $\ell:(X,X)\to (X)$ be the deterministic kernel $\ell(dy\mid (x_1,x_2))=\delta_{10x_1+x_2}(dy)$. Thus, the underlying deterministic function sends $(1,2)$ to $12$ and sends $(2,1)$ to $21$.

Let $D$ be the finite acyclic KSC diagram obtained by connecting the output of $k_1$ to the first input slot of $\ell$, and the output of $k_2$ to the second input slot of $\ell$. The diagram has empty external input profile and external output profile $(X)$. Its trace kernel $K_D:()\to (X)$ satisfies
\[
K_D(E\mid *)
=
\int_{\mathbb{R}}
\int_{\mathbb{R}}
\int_{\mathbb{R}}
\mathbf{1}_E(y)\,
\ell(dy\mid (x_1,x_2))\,
\delta_1(dx_1)\,
\delta_2(dx_2)
=
\delta_{12}(E)
\]
for every $E\in \mathcal{B}(\mathbb{R})$. This value is independent of whether the topological evaluation order evaluates $k_1$ before $k_2$ or $k_2$ before $k_1$. The input map for $\ell$ uses the specified slot indices, so the first input of $\ell$ is the output of $k_1$ and the second input of $\ell$ is the output of $k_2$. If evaluation order were incorrectly used to assemble the input tuple of $\ell$, then the order $(k_2,k_1,\ell)$ would give the input $(2,1)$ and the output $21$. Thus, slot indices are part of the data of the diagram.
\end{example}

\subsection{Structural laws for KSC}
\label{subsec:structural_laws_ksc}

We now relate the binary formula of Definition~\ref{def:ksc} to the trace semantics of Subsection~\ref{subsec:trace_semantics_ksc}. The structural laws for KSC then follow from Theorem~\ref{thm:trace_semantics_ksc}.

\begin{proposition}[Binary KSC agrees with trace semantics]
\label{prop:binary_ksc_trace_semantics}
Let $k:\mathbf{A}\to \mathbf{B}$ and $l:\mathbf{C}\to \mathbf{D}$ be morphisms in a Markov polycategory datum $\PC$. Suppose that $B_i=C_j$. Let $D_{k,l}^{i,j}$ be the finite acyclic KSC diagram with two vertices labelled by $k$ and $l$, and with one internal wire connecting the $i$-th output slot of $k$ to the $j$-th input slot of $l$. Give $D_{k,l}^{i,j}$ the external input profile $\boldsymbol{\Gamma}$ and external output profile $\boldsymbol{\Delta}$ of Definition~\ref{def:ksc}. Then the trace kernel of $D_{k,l}^{i,j}$ is the kernel $l\circ_{(i,j)}k$ of Definition~\ref{def:ksc}.
\end{proposition}

\begin{proof}
The diagram $D_{k,l}^{i,j}$ has the topological ordering in which the vertex labelled by $k$ comes first and the vertex labelled by $l$ comes second. Let the external input be $V_j(\mathbf{a},\mathbf{c}_{-j})\in \underline{\boldsymbol{\Gamma}}$.

The first vertex has input $\mathbf{a}$ and output $\mathbf{b}\in \underline{\mathbf{B}}$, distributed according to $k(d\mathbf{b}\mid \mathbf{a})$. The second vertex receives as input the tuple $T_j(\mathbf{c}_{-j},\pi_i(\mathbf{b}))\in \underline{\mathbf{C}}$. Its output $\mathbf{d}\in \underline{\mathbf{D}}$ is therefore distributed according to $l(d\mathbf{d}\mid T_j(\mathbf{c}_{-j},\pi_i(\mathbf{b})))$. The external output selection map sends $(\mathbf{b},\mathbf{d})$ to $U_i(\pi_{-i}(\mathbf{b}),\mathbf{d})\in \underline{\boldsymbol{\Delta}}$.

Hence, the trace kernel of $D_{k,l}^{i,j}$ assigns to $E\in \Sigma_{\underline{\boldsymbol{\Delta}}}$ the value
\[
\int_{\underline{\mathbf{B}}}
\left(
\int_{\underline{\mathbf{D}}}
\mathbf{1}_E\left(U_i(\pi_{-i}(\mathbf{b}),\mathbf{d})\right)
\,l(d\mathbf{d}\mid T_j(\mathbf{c}_{-j},\pi_i(\mathbf{b})))
\right)
k(d\mathbf{b}\mid \mathbf{a}).
\]
This is the formula in Definition~\ref{def:ksc}. Therefore, the trace kernel of $D_{k,l}^{i,j}$ is $l\circ_{(i,j)}k$.
\end{proof}

\begin{corollary}[Well-definedness of KSC]
\label{cor:ksc_well_defined}
Let $k:\mathbf{A}\to \mathbf{B}$ and $l:\mathbf{C}\to \mathbf{D}$ be morphisms in a Markov polycategory datum $\PC$. If $B_i=C_j$, then the kernel slotwise composite $l\circ_{(i,j)}k:\boldsymbol{\Gamma}\to \boldsymbol{\Delta}$ of Definition~\ref{def:ksc} is a Markov kernel.
\end{corollary}

\begin{proof}
By Proposition~\ref{prop:binary_ksc_trace_semantics}, the binary KSC $l\circ_{(i,j)}k$ is the trace kernel of the finite acyclic diagram $D_{k,l}^{i,j}$. By Theorem~\ref{thm:trace_semantics_ksc}, every finite acyclic KSC diagram determines a Markov kernel. Hence $l\circ_{(i,j)}k$ is a Markov kernel.
\end{proof}

\begin{corollary}[Unitality of KSC]
\label{cor:ksc_unitality}
Let $h:\mathbf{A}\to \mathbf{B}$ be a morphism in a Markov polycategory datum $\PC$, where $\mathbf{A}=(A_1,\ldots,A_m)$ and $\mathbf{B}=(B_1,\ldots,B_n)$. Then KSC satisfies the following unit laws. For every input slot $j\in \{1,\ldots,m\}$, $h\circ_{(1,j)}\id_{A_j}=h$. For every output slot $i\in \{1,\ldots,n\}$, $\id_{B_i}\circ_{(i,1)}h=h$.
\end{corollary}

\begin{proposition}[Compatibility of KSC with trace semantics]
\label{prop:ksc_compatible_trace_semantics}
Let $D_1$ and $D_2$ be finite acyclic KSC diagrams in a Markov polycategory datum $\PC$. Suppose that the $i$-th external output object of $D_1$ is equal to the $j$-th external input object of $D_2$. Let $D$ be the finite acyclic diagram obtained by connecting this output slot of $D_1$ to this input slot of $D_2$, with external input and output profiles formed by the insertion convention of Definition~\ref{def:ksc}. Then $K_D = K_{D_2}\circ_{(i,j)}K_{D_1}$, where the right-hand side denotes the KSC integral formula of Definition~\ref{def:ksc} applied to the trace kernels $K_{D_1}$ and $K_{D_2}$.
\end{proposition}

To state associativity and interchange, we distinguish diagrams from their binary evaluations. A finite acyclic KSC diagram has a trace kernel by Theorem~\ref{thm:trace_semantics_ksc}. An iterated KSC expression is auxiliary syntax recording one way of evaluating such a diagram by successive binary KSC operations.

\begin{definition}[Iterated KSC expression]
\label{def:iterated_ksc_expression}
Let $D$ be a finite acyclic KSC diagram. An iterated KSC expression over $D$ is defined inductively. Each such expression $E$ has an associated subdiagram $D_E$ of $D$ and an evaluation kernel $\mathbf{I}_{D_E}\to \mathbf{O}_{D_E}$.

First, for every vertex $u\in V(D)$, the single vertex $u$ is an iterated KSC expression. The associated subdiagram $D_u$ is the one-vertex subdiagram labelled by $k_u:\mathbf{A}_u\to \mathbf{B}_u$. The evaluation kernel of this expression is $k_u$.

Second, let $E_1$ and $E_2$ be iterated KSC expressions over $D$ whose associated subdiagrams $D_{E_1}$ and $D_{E_2}$ have disjoint vertex sets. Suppose that an internal wire $e=((u,p),(v,q))$ of $D$ connects an external output slot of $D_{E_1}$ to an external input slot of $D_{E_2}$. Let $r$ be the position of this external output slot in the ordered output profile $\mathbf{O}_{D_{E_1}}$, and let $s$ be the position of this external input slot in the ordered input profile $\mathbf{I}_{D_{E_2}}$. Then a new iterated KSC expression $E$ is formed from the data $(E_1,E_2,e)$. Its associated subdiagram $D_E$ is obtained by connecting $D_{E_1}$ to $D_{E_2}$ along the wire $e$, with external profiles formed by the insertion convention of Definition~\ref{def:ksc}. Its evaluation kernel is the KSC composite $H_2\circ_{(r,s)}H_1$, where $H_1:\mathbf{I}_{D_{E_1}}\to \mathbf{O}_{D_{E_1}}$ and $H_2:\mathbf{I}_{D_{E_2}}\to \mathbf{O}_{D_{E_2}}$ are the evaluation kernels of $E_1$ and $E_2$, respectively.

An iterated KSC expression $E$ realizes $D$ if $D_E=D$, including the ordered external input and output profiles.
\end{definition}

\begin{corollary}[Associativity and interchange for KSC]
\label{cor:ksc_associativity_interchange}
Let $D$ be a finite acyclic KSC diagram in a Markov polycategory datum $\PC$. If $E$ is an iterated KSC expression realizing $D$, then the evaluation kernel of $E$ is the trace kernel $K_D$. Consequently, any two iterated KSC expressions realizing the same diagram $D$ have the same evaluation kernel. In this sense, KSC satisfies the associativity and interchange laws for ordered polycategorical composition.
\end{corollary}

\begin{proof}
We prove the statement by induction on the construction of the iterated KSC expression $E$.

If $E$ is a single vertex $u$, then $D_E$ is the one-vertex diagram labelled by $k_u$. The evaluation kernel of $E$ is $k_u$. The trace kernel of $D_E$ is also $k_u$. Hence the evaluation kernel of $E$ is $K_{D_E}$.

For the induction step, suppose that $E$ is formed from the data $(E_1,E_2,e)$, where $e$ connects the $r$-th external output slot of $D_{E_1}$ to the $s$-th external input slot of $D_{E_2}$. Let $H_1$ and $H_2$ be the evaluation kernels of $E_1$ and $E_2$, respectively. By definition, the evaluation kernel of $E$ is $H_2\circ_{(r,s)}H_1$. By the induction hypothesis, $H_1=K_{D_{E_1}}$, $H_2=K_{D_{E_2}}$. Therefore, the evaluation kernel of $E$ is $K_{D_{E_2}}\circ_{(r,s)}K_{D_{E_1}}$. By Proposition~\ref{prop:ksc_compatible_trace_semantics}, this kernel is the trace kernel of the subdiagram obtained by connecting $D_{E_1}$ and $D_{E_2}$ along $e$. This subdiagram is $D_E$. Hence, the evaluation kernel of $E$ is $K_{D_E}$.

If $E$ realizes $D$, then $D_E=D$, so the evaluation kernel of $E$ is $K_D$. If $E$ and $E'$ are two iterated KSC expressions realizing the same diagram $D$, then both evaluation kernels are equal to $K_D$. This proves the claim.
\end{proof}

\begin{definition}[Markov polycategory]
\label{def:markov_polycategory}
A Markov polycategory is a Markov polycategory datum $\PC$ such that, if $k:\mathbf{A}\to \mathbf{B}$ and $l:\mathbf{C}\to \mathbf{D}$ are morphisms of $\PC$ and $B_i=C_j$, then the KSC kernel $l\circ_{(i,j)}k$ of Definition~\ref{def:ksc} belongs to $\Hom_{\PC}(\boldsymbol{\Gamma};\boldsymbol{\Delta})$, where $\boldsymbol{\Gamma}$ and $\boldsymbol{\Delta}$ are the ordered input and output profiles of Definition~\ref{def:ksc}.

The slotwise composition operation of a Markov polycategory is this KSC operation. The identity morphisms are the identity Markov kernels.
\end{definition}

\begin{corollary}[KSC gives an ordered polycategory of Markov kernels]
\label{cor:ksc_gives_markov_polycategory}
Let $\PC$ be a Markov polycategory (Definition~\ref{def:markov_polycategory}). Then $\PC$, equipped with KSC and identity kernels, satisfies the unit, associativity, and interchange axioms of an ordered polycategory.
\end{corollary}

\begin{proof}
The unit laws are Corollary~\ref{cor:ksc_unitality}. The associativity and interchange laws are Corollary~\ref{cor:ksc_associativity_interchange}. Closure under KSC is part of Definition~\ref{def:markov_polycategory}. Hence, the Markov polycategory satisfies the ordered polycategory axioms with KSC as slotwise composition.
\end{proof}

\begin{figure}[t!]
\centering
\begin{subfigure}[b]{0.32\textwidth}
\centering
\includegraphics[width=\textwidth]{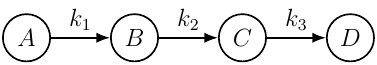}
\caption{}
\label{fig:ksc_bn_chain_struct_dag}
\end{subfigure}
\hfill
\begin{subfigure}[b]{0.32\textwidth}
\centering
\includegraphics[width=\textwidth]{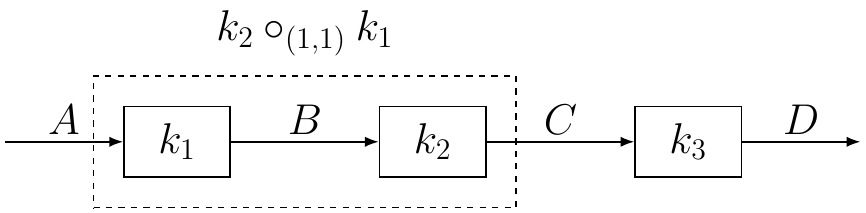}
\caption{}
\label{fig:ksc_bn_chain_struct_right_assoc}
\end{subfigure}
\hfill
\begin{subfigure}[b]{0.32\textwidth}
\centering
\includegraphics[width=\textwidth]{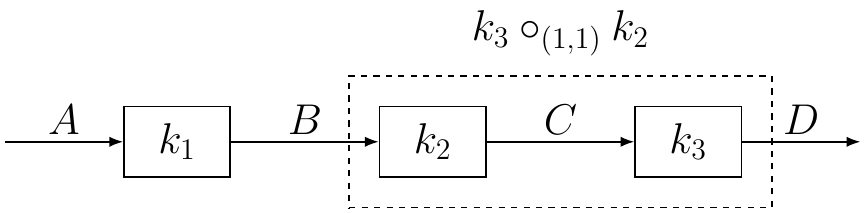}
\caption{}
\label{fig:ksc_bn_chain_struct_left_assoc}
\end{subfigure}
\begin{subfigure}[b]{0.32\textwidth}
\centering
\includegraphics[width=0.65\textwidth]{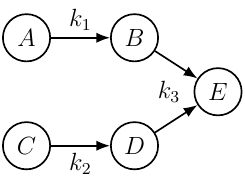}
\caption{}
\label{fig:ksc_bn_v_struct_dag}
\end{subfigure}
\hfill
\begin{subfigure}[b]{0.32\textwidth}
\centering
\includegraphics[width=\textwidth]{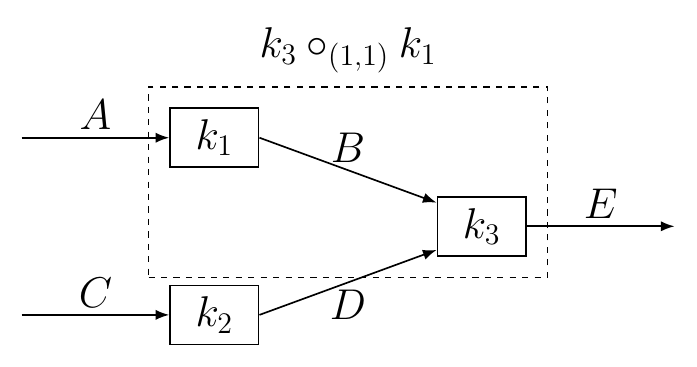}
\caption{}
\label{fig:ksc_bn_v_struct_assoc01}
\end{subfigure}
\hfill
\begin{subfigure}[b]{0.32\textwidth}
\centering
\includegraphics[width=\textwidth]{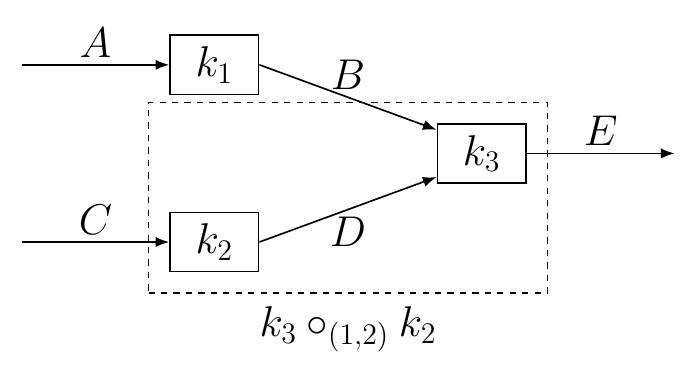}
\caption{}
\label{fig:ksc_bn_v_struct_assoc02}
\end{subfigure}
\caption{Finite Bayesian networks and their KSC string-diagram reductions. 
(a) A chain with kernels $k_1:(A)\to(B)$, $k_2:(B)\to(C)$, and $k_3:(C)\to(D)$. 
(b) String diagram highlighting the reduction $k_3\circ_{(1,1)}(k_2\circ_{(1,1)}k_1)$. 
(c) String diagram highlighting the reduction $(k_3\circ_{(1,1)}k_2)\circ_{(1,1)}k_1$. 
Both reductions denote the same trace kernel $K:(A)\to(D)$. 
(d) A v-structure with kernels $k_1:(A)\to(B)$, $k_2:(C)\to(D)$, and $k_3:(B,D)\to(E)$. 
(e) String diagram highlighting the reduction $(k_3\circ_{(1,1)}k_1)\circ_{(1,2)}k_2$. 
(f) String diagram highlighting the reduction $(k_3\circ_{(1,2)}k_2)\circ_{(1,1)}k_1$. 
Both reductions denote the same trace kernel $H:(A,C)\to(E)$.}
\label{fig:ksc_bn}
\end{figure}

\begin{example}[Trace semantics for finite Bayesian networks]
\label{ex:ksc_bayesian_network}
The trace semantics of Subsection~\ref{subsec:trace_semantics_ksc} recovers the marginalization formulas for finite Bayesian networks. Two examples of finite Bayesian networks and their corresponding KSC string-diagram reductions are shown in Figure~\ref{fig:ksc_bn}.

First consider a chain of three kernels $k_1:(A)\to (B)$, $k_2:(B)\to (C)$, $k_3:(C)\to (D)$, where all objects have finite underlying sets, equipped with their power-set $\sigma$-algebras. Write $M_1(b\mid a)=k_1(\{b\}\mid a)$, $M_2(c\mid b)=k_2(\{c\}\mid b)$, $M_3(d\mid c)=k_3(\{d\}\mid c)$. The finite acyclic KSC diagram obtained by connecting the output wire $B$ to the input of $k_2$ and the output wire $C$ to the input of $k_3$ has external input profile $(A)$ and external output profile $(D)$. Its trace kernel $K:(A)\to (D)$ is
\[
K(\{d\}\mid a)
=
\sum_{b\in \underline{B}}
\sum_{c\in \underline{C}}
M_3(d\mid c)M_2(c\mid b)M_1(b\mid a).
\]
This expression is independent of whether the diagram is evaluated as $k_3\circ_{(1,1)}(k_2\circ_{(1,1)}k_1)$ or as $(k_3\circ_{(1,1)}k_2)\circ_{(1,1)}k_1$. Both iterated KSC expressions denote the same trace kernel $K$, by Corollary~\ref{cor:ksc_associativity_interchange}.

As a second example, consider a v-structure $k_1:(A)\to (B)$, $k_2:(C)\to (D)$, $k_3:(B,D)\to (E)$, again with finite underlying sets. Let $M_1(b\mid a)=k_1(\{b\}\mid a)$, $M_2(d\mid c)=k_2(\{d\}\mid c)$, $M_3(e\mid b,d)=k_3(\{e\}\mid (b,d))$. The finite acyclic KSC diagram obtained by connecting the output of $k_1$ to the first input of $k_3$ and the output of $k_2$ to the second input of $k_3$ has external input profile $(A,C)$ and external output profile $(E)$. Its trace kernel $H:(A,C)\to (E)$ is
\[
H(\{e\}\mid (a,c))
=
\sum_{b\in \underline{B}}
\sum_{d\in \underline{D}}
M_3(e\mid b,d)M_1(b\mid a)M_2(d\mid c).
\]
The two iterated composites $(k_3\circ_{(1,1)}k_1)\circ_{(1,2)}k_2$ and $(k_3\circ_{(1,2)}k_2)\circ_{(1,1)}k_1$ therefore denote the same kernel $H$. This equality is the finite-sum instance of the interchange law. It expresses that the two internal variables $b$ and $d$ are marginalized out, and that the order of these finite marginalizations does not affect the resulting kernel.
\end{example}

\section{Colored Markov polycategories}
\label{sec:colored_markov_polycategories}

This section adds typed structure to Markov polycategories. Object colors classify the types of wires. Morphism colors classify the types of boxes. Typed connections between different but compatible object colors are implemented by interface kernels.

\subsection{Object colors and morphism colors}
\label{subsec:object_morphism_colors}

The coloring structure used in this paper has two levels. Object colors are organized by a category $\K$. Morphism colors are organized by an ordered polycategory $\M$ whose objects are the object colors. Thus, morphism colors have typed input and output profiles.

\begin{definition}[Color system]
\label{def:color_system}
A color system $(\K,\M,\iota)$ consists of the following data.
\begin{enumerate}
\item There is a category $\K$, called the object-color category. Its objects are called object colors.
\item There is an ordered polycategory $\M$, called the morphism-color polycategory, such that $\Ob(\M)=\Ob(\K)$. The morphisms of $\M$ are called morphism colors.
\item Let $\M_1$ denote the category obtained by restricting $\M$ to unary profiles. Thus, $\Ob(\M_1)=\Ob(\M)$ and $\Hom_{\M_1}(c,d)=\Hom_{\M}((c);(d))$. Composition in $\M_1$ is the unary slotwise composition of $\M$. There is an identity-on-objects functor $\iota:\K\to \M_1$. Thus, $\iota$ assigns to each object-color morphism $f:c\to d$ in $\K$ a unary morphism color $\iota(f):(c)\to (d)$ in $\M$.
\end{enumerate}
\end{definition}

The functor $\iota$ records how object-color conversions are seen at the level of morphism colors. Since $\iota$ is a functor, it preserves identities and composition, that is, $\iota(\id_c)=\id_c$, $\iota(g\circ f)=\iota(g)\circ_{(1,1)}\iota(f)$.

Let $\PC$ be a Markov polycategory datum. An object-coloring of $\PC$ is a function $\chi:\Ob(\PC)\to \Ob(\K)$. If $\mathbf{A}=(A_1,\ldots,A_m)$ is an ordered tuple of objects of $\PC$, then we write $\chi(\mathbf{A})=(\chi(A_1),\ldots,\chi(A_m))$.

A morphism-coloring of $\PC$ assigns to each morphism $k:\mathbf{A}\to \mathbf{B}$ of $\PC$ a morphism color $\psi(k)\in \Hom_{\M}(\chi(\mathbf{A});\chi(\mathbf{B}))$. Thus, the color of a morphism records its role, and the colors of its input and output profiles. The identity kernel $\id_X:(X)\to (X)$ is required to have morphism color $\psi(\id_X)=\id_{\chi(X)}$.

Using a morphism-color polycategory lets the color of a composite be computed by the same slotwise operation used for kernels. If $k:\mathbf{A}\to \mathbf{B}$ and $l:\mathbf{C}\to \mathbf{D}$ are composed along compatible slots, then the corresponding morphism colors are composed by the same slotwise operation in $\M$. This avoids introducing a separate ad hoc operation for coloring composites.

\subsection{Interface systems}
\label{subsec:interface_systems}

Object-color compatibility is realized at the level of measurable spaces by interface kernels. The purpose of an interface kernel is to turn a color-level compatibility morphism into a Markov kernel between the objects carried by the connected wires.

\begin{definition}[Interface data]
\label{def:interface_data}
Let $\PC$ be a Markov polycategory datum equipped with a color system $(\K,\M,\iota)$, an object-coloring $\chi$, and a morphism-coloring $\psi$. Interface data on $\PC$ consists of the following data.
\begin{enumerate}
\item For every ordered pair of objects $B,C\in \Ob(\PC)$, there is a set $\Int_{\PC}(B,C)\subseteq \Hom_{\K}(\chi(B),\chi(C))$. An element $f\in \Int_{\PC}(B,C)$ is called an interface color morphism from $B$ to $C$.
\item For every $f\in \Int_{\PC}(B,C)$, there is a unary morphism $\kappa_{f;B,C}\in \Hom_{\PC}((B);(C))$. This morphism is called the interface kernel determined by $f$ at the object pair $(B,C)$. As a Markov kernel, it has signature $\kappa_{f;B,C}:(\Sigma_{\underline{C}},\underline{B})\to [0,1]$.
\end{enumerate}
As a compatibility condition, the morphism color of the interface kernel is required to be the unary morphism color determined by $f$, that is, $\psi(\kappa_{f;B,C})=\iota(f)\in \Hom_{\M}((\chi(B));(\chi(C)))$.
\end{definition}

A typed connection from an output object $B$ to an input object $C$ is therefore specified by the existence of a morphism $\chi(B)\to \chi(C)$ in $\K$, and by an interface color morphism $f\in \Int_{\PC}(B,C)$, together with the corresponding kernel $\kappa_{f;B,C}$. This avoids ambiguity when there is more than one color morphism between the same two object colors.

The notation $\kappa_{f;B,C}$ includes the objects $B$ and $C$. This is necessary because the same color-level morphism $f$ may be used for different pairs of objects with different underlying standard Borel spaces. The color morphism records typed compatibility, while the interface kernel records the corresponding data transformation.

\subsection{Identity and composition laws for interface systems}
\label{subsec:interface_coherence}

Interface kernels must be compatible with identity and composition in the object-color category. Otherwise, different bracketings of a typed diagram could insert different interface behavior.

\begin{definition}[Interface system]
\label{def:interface_system}
Let $\PC$ be a Markov polycategory equipped with interface data as in Definition~\ref{def:interface_data}. The interface data form an interface system if the following two laws hold.
\begin{enumerate}
\item The identity law holds. For every object $X\in \Ob(\PC)$, $\id_{\chi(X)}\in \Int_{\PC}(X,X)$, and the corresponding interface kernel is the identity kernel: $\kappa_{\id_{\chi(X)};X,X}=\id_X$.
\item The composition law holds. If $B,C,D\in \Ob(\PC)$, $f\in \Int_{\PC}(B,C)$, and $g\in \Int_{\PC}(C,D)$, then $g\circ f\in \Int_{\PC}(B,D)$, and the corresponding interface kernel satisfies
\[
\kappa_{g\circ f;B,D}
=
\kappa_{g;C,D}\circ_{(1,1)}\kappa_{f;B,C}.
\]
\end{enumerate}
\end{definition}

Since the interface kernels are unary kernels, the composite
$\kappa_{g;C,D}\circ_{(1,1)}\kappa_{f;B,C}$ is the Chapman-Kolmogorov composite of Markov kernels. Thus, the composition law says that performing two compatible type conversions in sequence is the same as performing the single interface conversion associated with their composite color morphism.

The coloring condition in Definition~\ref{def:interface_data} is compatible with Definition~\ref{def:interface_system}. Indeed, since $\iota:\K\to \M_1$ is a functor,
\[
\psi(\kappa_{g\circ f;B,D})
=
\iota(g\circ f)
=
\iota(g)\circ_{(1,1)}\iota(f)
=
\psi(\kappa_{g;C,D})\circ_{(1,1)}\psi(\kappa_{f;B,C}).
\]
Thus, the identity and composition laws for interface kernels are compatible with the corresponding laws for morphism colors.

\begin{proposition}[Composition law for interface paths]
\label{prop:interface_path_coherence}
Let $\PC$ be a Markov polycategory equipped with an interface system. Let $X_0,X_1,\ldots,X_r$ be objects of $\PC$, and let $f_s\in \Int_{\PC}(X_{s-1},X_s)$ for $s=1,\ldots,r$, where $r \ge 1$. Then the composite color morphism $f_r\circ\cdots\circ f_1:\chi(X_0)\to \chi(X_r)$ is an interface color morphism from $X_0$ to $X_r$, and its interface kernel is
\[
\kappa_{f_r\circ\cdots\circ f_1;X_0,X_r}
=
\kappa_{f_r;X_{r-1},X_r}
\circ_{(1,1)}
\cdots
\circ_{(1,1)}
\kappa_{f_1;X_0,X_1}.
\]
The right-hand side is independent of the bracketing.
\end{proposition}

\begin{proof}
The proof is by induction on $r$. The case $r=1$ is immediate.

Assume that $r>1$ and that the statement holds for paths of length $r-1$. Let $f_s\in \Int_{\PC}(X_{s-1},X_s)$ for $s=1,\ldots,r$. By the induction hypothesis, $f_{r-1}\circ\cdots\circ f_1$ is an interface color morphism from $X_0$ to $X_{r-1}$, and its interface kernel is the corresponding composite of $\kappa_{f_1},\ldots,\kappa_{f_{r-1}}$. Applying the composition law in Definition~\ref{def:interface_system} to this interface color morphism and $f_r$ shows that $f_r\circ\cdots\circ f_1$ belongs to $\Int_{\PC}(X_0,X_r)$, and that its interface kernel is obtained by composing $\kappa_{f_r;X_{r-1},X_r}$ with the interface kernel for $f_{r-1}\circ\cdots\circ f_1$.

The bracketing independence follows from Corollary~\ref{cor:ksc_associativity_interchange}, applied to the unary chain of interface kernels.
\end{proof}

The empty path case is governed by the identity law in Definition~\ref{def:interface_system}. If no interface color morphisms are present, the composite from $X_0$ to itself is understood to be $\id_{\chi(X_0)}$, and the corresponding interface kernel is $\id_{X_0}$.

\subsection{Definition of colored Markov polycategory}
\label{subsec:colored_markov_polycategories}

We now collect the color data, morphism-coloring, and interface system into one structure. The colored slotwise composition operation itself is defined later, in Section~\ref{sec:cksc_diagram_semantics}, using interface kernels.

\begin{definition}[Colored Markov polycategory]
\label{def:colored_markov_polycategory}
A colored Markov polycategory (CMP) consists of the following data.
\begin{enumerate}
\item There is a Markov polycategory $\PC$.
\item There is a color system $(\K,\M,\iota)$.
\item There is an object-coloring function $\chi : \Ob(\PC) \to \Ob(\K)$.
\item There is a morphism-coloring function $\psi$ which assigns to every morphism $k:\mathbf{A}\to \mathbf{B}$ of $\PC$ a morphism color $\psi(k) \in \Hom_{\M}(\chi(\mathbf{A}); \chi(\mathbf{B}))$, where $\chi(\mathbf{A})=(\chi(A_1),\ldots,\chi(A_m))$ for $\mathbf{A}=(A_1,\ldots,A_m)$, and similarly for $\mathbf{B}$.
\item There is an interface system on $\PC$, relative to the color system $(\K,\M,\iota)$ and the colorings $\chi$ and $\psi$.
\end{enumerate}

The pair $(\chi,\psi)$ is required to preserve identities and KSC. This means that, for every object $X\in \Ob(\PC)$, $\psi(\id_X)=\id_{\chi(X)}$. Moreover, if $k:\mathbf{A}\to \mathbf{B}$ and $l:\mathbf{C}\to \mathbf{D}$ are morphisms of $\PC$, and if $i$ is an output slot of $k$ and $j$ is an input slot of $l$ such that $B_i=C_j$, then
\[
\psi(l\circ_{(i,j)}k)
=
\psi(l)\circ_{(i,j)}\psi(k).
\]
\end{definition}

Thus, a CMP is a Markov polycategory whose kernels are typed by object colors and morphism colors, and whose specified type conversions are represented by interface kernels. The requirement that $(\chi,\psi)$ preserves KSC says that the color of an underlying KSC composite is determined by composing the corresponding morphism colors in $\M$. The interface system is a mechanism for composing kernels whose connected slots have different but compatible object colors. This typed composition is introduced in Section~\ref{sec:cksc_diagram_semantics}.

\subsection{CMP-functors}
\label{subsec:cmp_functors}

The co-indexed construction in Section~\ref{sec:ccmp} requires structure-preserving maps between CMPs. We use a fixed color system when comparing CMPs. Thus, colors are treated as a global type system shared by the CMPs under comparison.

\begin{definition}[CMP-functor]
\label{def:cmp_functor}
Let $\PC$ and $\PC'$ be CMPs over the same color system $(\K,\M,\iota)$. Write their object-coloring functions as $\chi,\chi'$ and their morphism-coloring functions as $\psi,\psi'$. A CMP-functor $G:\PC\to \PC'$ consists of the following data.
\begin{enumerate}
\item $G$ assigns to each object $X\in \Ob(\PC)$ an object $G X\in \Ob(\PC')$. For an ordered tuple $\mathbf{A}=(A_1,\ldots,A_m)$, write $G\mathbf{A}=(G A_1,\ldots,G A_m)$.
\item For every morphism $k:\mathbf{A}\to \mathbf{B}$ in $\PC$, $G$ assigns a morphism $Gk:G\mathbf{A}\to G\mathbf{B}$ in $\PC'$.
\end{enumerate}

The assignments are required to preserve identities and KSC. This means that, for every object $X\in \Ob(\PC)$, $G(\id_X)=\id_{G X}$. Moreover, if $l\circ_{(i,j)}k$ is defined in $\PC$, then $G(l\circ_{(i,j)}k)=G l\circ_{(i,j)}G k$.

The assignments are also required to preserve colors. This means that, for every object $X\in \Ob(\PC)$, $\chi'(G X)=\chi(X)$, and, for every morphism $k$ of $\PC$, $\psi'(Gk)=\psi(k)$.

Finally, the assignments are required to preserve interface systems. If $B,C\in \Ob(\PC)$ and $f\in \Int_{\PC}(B,C)$, then $f\in \Int_{\PC'}(G B,G C)$, and $G(\kappa_{f;B,C})=\kappa'_{f;G B,G C}$. Here, $\kappa$ and $\kappa'$ denote the interface kernels of $\PC$ and $\PC'$, respectively.
\end{definition}

Thus, a CMP-functor is an ordered polyfunctor between the underlying Markov polycategories which also preserves object colors, morphism colors, and typed interface kernels. The condition on interfaces ensures that applying $G$ to a typed connection gives the same result as first transporting the connected objects and then applying the corresponding interface kernel in the target CMP.

CMPs over a fixed color system and CMP-functors between them form a category, denoted by $\mathbf{CMP}_{\K,\M,\iota}$, with identities and composition defined component-wise (see Proposition~\ref{prop:category_of_cmps}). When the color system is clear from context, this category is denoted by $\mathbf{CMP}$.

\section{Colored kernel slotwise composition and diagram semantics}
\label{sec:cksc_diagram_semantics}

This section defines typed slotwise composition in a CMP. The uncolored KSC operation requires equality of the connected objects. In a CMP, this condition is replaced by typed compatibility specified by an interface color morphism. If an output object $B$ and an input object $C$ are connected by an interface color morphism $f\in \Int_{\PC}(B,C)$, the connection is mediated by the interface kernel $\kappa_{f;B,C}:B\to C$. Proofs not included in the main text of this section are given in Appendix~\ref{app:cksc_proofs}.

\subsection{Colored kernel slotwise composition}
\label{subsec:cksc}

Let $\PC$ be a CMP. Let $k:\mathbf{A}\to \mathbf{B}$ and $l:\mathbf{C}\to \mathbf{D}$ be morphisms of the underlying Markov polycategory, where $\mathbf{A}=(A_1,\ldots,A_m)$, $\mathbf{B}=(B_1,\ldots,B_n)$, $\mathbf{C}=(C_1,\ldots,C_q)$, $\mathbf{D}=(D_1,\ldots,D_r)$. Let $i\in \{1,\ldots,n\}$ and $j\in \{1,\ldots,q\}$. A colored connection from the $i$-th output of $k$ to the $j$-th input of $l$ is specified by an interface color morphism $f\in \Int_{\PC}(B_i,C_j)$. The associated interface kernel is $\kappa_{f;B_i,C_j}:(B_i)\to (C_j)$.

\begin{definition}[Colored kernel slotwise composition]
\label{def:cksc}
The colored kernel slotwise composition (CKSC) of $l$ with $k$ along the slot pair $(i,j)$ and interface color morphism $f\in \Int_{\PC}(B_i,C_j)$ is the morphism $l\circ^{f}_{(i,j)}k:\boldsymbol{\Gamma}\to \boldsymbol{\Delta}$ defined by
\[
l\circ^{f}_{(i,j)}k
=
l\circ_{(i,j)}
\bigl(
\kappa_{f;B_i,C_j}\circ_{(i,1)}k
\bigr).
\]
The input and output profiles are
\[
\boldsymbol{\Gamma}
=
(C_1,\ldots,C_{j-1},A_1,\ldots,A_m,C_{j+1},\ldots,C_q)
\]
and
\[
\boldsymbol{\Delta}
=
(B_1,\ldots,B_{i-1},D_1,\ldots,D_r,B_{i+1},\ldots,B_n).
\]
\end{definition}

Equivalently, the CKSC kernel is given by the integral formula
\[
\begin{split}
&(l\circ^{f}_{(i,j)}k)(E\mid V_j(\mathbf{a},\mathbf{c}_{-j}))\\
&\qquad =
\int_{\underline{\mathbf{B}}}
\left(
\int_{\underline{C_j}}
\left(
\int_{\underline{\mathbf{D}}}
\mathbf{1}_E\bigl(U_i(\pi_{-i}(\mathbf{b}),\mathbf{d})\bigr)
\,l(d\mathbf{d}\mid T_j(\mathbf{c}_{-j},z))
\right)
\kappa_{f;B_i,C_j}(dz\mid \pi_i(\mathbf{b}))
\right)
k(d\mathbf{b}\mid \mathbf{a})
\end{split}
\]
for every $E\in \Sigma_{\underline{\boldsymbol{\Delta}}}$ and every input $V_j(\mathbf{a},\mathbf{c}_{-j})\in \underline{\boldsymbol{\Gamma}}$. When $B_i=C_j$ and $f=\id_{\chi(B_i)}$, the interface kernel is $\id_{B_i}$ by Definition~\ref{def:interface_system}. In that case CKSC reduces to the uncolored KSC of Definition~\ref{def:ksc}.

The morphism color of the CKSC composite is determined by the morphism-color polycategory. Since $\psi(\kappa_{f;B_i,C_j})=\iota(f)$, it follows that
\[
\psi(l\circ^{f}_{(i,j)}k)
=
\psi(l)\circ_{(i,j)}
\left(
\iota(f)\circ_{(i,1)}\psi(k)
\right).
\]
Thus, the typed data transformation and the morphism color of the composite are compatible by construction.

\subsection{CKSC string diagrams}
\label{subsec:cksc_string_diagrams}

A CKSC string diagram is a finite acyclic string diagram whose internal connections are typed by interface color morphisms. The interface color morphism, together with the connected source and target objects, determines the kernel inserted on the connected wire.

\begin{definition}[Finite acyclic CKSC diagram]
\label{def:finite_acyclic_cksc_diagram}
Let $\PC$ be a CMP. A finite acyclic CKSC diagram $D$ in $\PC$ consists of the following data.
\begin{enumerate}
\item $D$ has a finite set $V(D)$ of vertices. Each vertex $u\in V(D)$ is labelled by a morphism $k_u:\mathbf{A}_u\to \mathbf{B}_u$ of the underlying Markov polycategory of $\PC$, where $\mathbf{A}_u=(A_{u,1},\ldots,A_{u,m_u})$ and $\mathbf{B}_u=(B_{u,1},\ldots,B_{u,n_u})$ are finite ordered tuples of objects.
\item $D$ has a finite set $W_{\mathrm{int}}(D)$ of internal wires. Each internal wire is a triple $e=((u,p),(v,q),f_e)$, where $u,v\in V(D)$, $1\leq p\leq n_u$, $1\leq q\leq m_v$, and $f_e\in \Int_{\PC}(B_{u,p},A_{v,q})$. The wire connects the $p$-th output slot of $u$ to the $q$-th input slot of $v$, and the interface color morphism $f_e$ determines the interface kernel $\kappa_{f_e;B_{u,p},A_{v,q}}:B_{u,p}\to A_{v,q}$.
\item Each input slot is the target of at most one internal wire, and each output slot is the source of at most one internal wire.
\item The directed graph on $V(D)$ with an edge $u\to v$ for every internal wire $((u,p),(v,q),f_e)$ is acyclic.
\item Let the external input slots of $D$ be the input slots $(v,q)$ that are not targets of internal wires. Let the external output slots of $D$ be the output slots $(u,p)$ that are not sources of internal wires. The diagram includes a specified ordering of all its external input slots and all its external output slots. If the ordered list of external input slots is $((v_1,q_1),\ldots,(v_a,q_a))$, then the external input profile of $D$ is $\mathbf{I}_D=(A_{v_1,q_1},\ldots,A_{v_a,q_a})$. If the ordered list of external output slots is $((u_1,p_1),\ldots,(u_b,p_b))$, then the external output profile of $D$ is $\mathbf{O}_D=(B_{u_1,p_1},\ldots,B_{u_b,p_b})$.
\end{enumerate}
\end{definition}

The third clause of Definition~\ref{def:finite_acyclic_cksc_diagram} is the linearity condition on wires. It excludes implicit copying and merging. Since every unconnected output slot is an external output slot, the diagram has no implicit discarding.

The acyclicity condition is imposed on the graph of non-interface vertices. Since every interface kernel lies on a wire directed from an output slot to an input slot, inserting these interface kernels as unary vertices does not create a directed cycle.

\begin{definition}[Interface expansion]
\label{def:interface_expansion}
Let $D$ be a finite acyclic CKSC diagram in a CMP $\PC$. The interface expansion of $D$, denoted by $\widehat{D}$, is the finite acyclic KSC diagram obtained as follows. For every vertex $u\in V(D)$, $\widehat{D}$ has a vertex labelled by the same kernel $k_u$. For every internal wire $e=((u,p),(v,q),f_e)$ of $D$, the diagram $\widehat{D}$ has an additional unary vertex, called the interface vertex associated with $e$, labelled by the interface kernel $\kappa_{f_e;B_{u,p},A_{v,q}}:B_{u,p}\to A_{v,q}$. The $p$-th output slot of $u$ is connected to the input slot of this interface vertex, and the output slot of the interface vertex is connected to the $q$-th input slot of $v$. The external input and output profiles of $\widehat{D}$ are the same as those of $D$.
\end{definition}

The interface expansion replaces each typed CKSC wire by a two-step KSC path through the corresponding interface kernel. Thus, a typed connection from $B$ to $C$ is interpreted by first applying the kernel $\kappa_{f;B,C}:B\to C$ and then feeding its output into the target input slot. This allows the trace semantics of CKSC diagrams to be defined using the KSC trace semantics established in Theorem~\ref{thm:trace_semantics_ksc}.

\subsection{Trace semantics for finite acyclic CKSC diagrams}
\label{subsec:trace_semantics_cksc}

The semantics of a CKSC diagram is defined by expanding its interface kernels and then applying the KSC trace semantics of Subsection~\ref{subsec:trace_semantics_ksc}. Definition~\ref{def:trace_kernel_cksc} states this construction.

\begin{definition}[Trace kernel of a CKSC diagram]
\label{def:trace_kernel_cksc}
Let $D$ be a finite acyclic CKSC diagram in a CMP $\PC$, and let $\widehat{D}$ be its interface expansion. The trace kernel of $D$ is the Markov kernel $K_D:\mathbf{I}_D\to \mathbf{O}_D$ defined by $K_D=K_{\widehat{D}}$, where $K_{\widehat{D}}$ is the KSC trace kernel of the finite acyclic KSC diagram $\widehat{D}$.
\end{definition}

\begin{theorem}[Trace semantics for finite acyclic CKSC diagrams]
\label{thm:trace_semantics_cksc}
Let $D$ be a finite acyclic CKSC diagram in a CMP $\PC$. Then $D$ determines a unique Markov kernel $K_D:\mathbf{I}_D\to \mathbf{O}_D$. This kernel is independent of the choice of topological ordering used to evaluate the interface expansion $\widehat{D}$.
\end{theorem}

\begin{proof}
By Definition~\ref{def:interface_expansion}, the interface expansion $\widehat{D}$ is a finite acyclic KSC diagram. Its external input and output profiles are the same as those of $D$. By Theorem~\ref{thm:trace_semantics_ksc}, the diagram $\widehat{D}$ determines a unique Markov kernel $K_{\widehat{D}}:\mathbf{I}_D\to \mathbf{O}_D$, independent of the topological ordering used to evaluate it. Definition~\ref{def:trace_kernel_cksc} sets $K_D=K_{\widehat{D}}$. Hence, $D$ determines a unique Markov kernel with external input profile $\mathbf{I}_D$ and external output profile $\mathbf{O}_D$.
\end{proof}

Theorem~\ref{thm:trace_semantics_cksc} says that typed diagrams can be evaluated by replacing each typed connection with its corresponding interface kernel. The interface expansion yields a KSC diagram whose internal wires are integrated out by the KSC trace semantics of Section~\ref{sec:markov_polycategories_ksc}.

\subsection{Structural laws for CKSC}
\label{subsec:structural_laws_cksc}

The structural laws for CKSC follow from the trace semantics of interface expansions. In this subsection, we first compare binary CKSC with trace semantics, and then derive well-definedness, unitality, associativity, and interchange for CKSC.

\begin{proposition}[Binary CKSC agrees with trace semantics]
\label{prop:binary_cksc_trace_semantics}
Let $k:\mathbf{A}\to \mathbf{B}$ and $l:\mathbf{C}\to \mathbf{D}$ be morphisms in a CMP $\PC$. Let $f\in \Int_{\PC}(B_i,C_j)$. Let $D_{k,l}^{i,j,f}$ be the finite acyclic CKSC diagram with two non-interface vertices labelled by $k$ and $l$, and with one internal wire connecting the $i$-th output slot of $k$ to the $j$-th input slot of $l$ through the interface color morphism $f$. Give $D_{k,l}^{i,j,f}$ the external input and output profiles $\boldsymbol{\Gamma}$ and $\boldsymbol{\Delta}$ of Definition~\ref{def:cksc}. Then the trace kernel of $D_{k,l}^{i,j,f}$ is the CKSC kernel $l\circ^f_{(i,j)}k$ of Definition~\ref{def:cksc}.
\end{proposition}

\begin{proof}
The interface expansion $\widehat{D}_{k,l}^{i,j,f}$ has three vertices. They are labelled by $k$, by the interface kernel $\kappa_{f;B_i,C_j}$, and by $l$. The output slot $i$ of $k$ is connected to the input of the interface kernel, and the output of the interface kernel is connected to the input slot $j$ of $l$.

By Proposition~\ref{prop:binary_ksc_trace_semantics} and Proposition~\ref{prop:ksc_compatible_trace_semantics}, the trace kernel of this expanded KSC diagram is
\[
l\circ_{(i,j)}
\left(
\kappa_{f;B_i,C_j}\circ_{(i,1)}k
\right).
\]
By Definition~\ref{def:cksc}, this is $l\circ^f_{(i,j)}k$. Since $K_{D_{k,l}^{i,j,f}}=K_{\widehat{D}_{k,l}^{i,j,f}}$ by Definition~\ref{def:trace_kernel_cksc}, the result follows.
\end{proof}

\begin{corollary}[Well-definedness of CKSC]
\label{cor:cksc_well_defined}
Let $k:\mathbf{A}\to \mathbf{B}$ and $l:\mathbf{C}\to \mathbf{D}$ be morphisms in a CMP $\PC$. If $f\in \Int_{\PC}(B_i,C_j)$, then the colored kernel slotwise composite $l\circ^f_{(i,j)}k:\boldsymbol{\Gamma}\to \boldsymbol{\Delta}$ of Definition~\ref{def:cksc} is a Markov kernel.
\end{corollary}

\begin{proof}
By Proposition~\ref{prop:binary_cksc_trace_semantics}, the CKSC composite $l\circ^f_{(i,j)}k$ is the trace kernel of the finite acyclic CKSC diagram $D_{k,l}^{i,j,f}$. By Theorem~\ref{thm:trace_semantics_cksc}, every finite acyclic CKSC diagram determines a Markov kernel. Hence, $l\circ^f_{(i,j)}k$ is a Markov kernel.
\end{proof}

\begin{corollary}[Unitality of CKSC]
\label{cor:cksc_unitality}
Let $h:\mathbf{A}\to \mathbf{B}$ be a morphism in a CMP $\PC$, where $\mathbf{A}=(A_1,\ldots,A_m)$ and $\mathbf{B}=(B_1,\ldots,B_n)$. Then CKSC satisfies the following unit laws.

For every input slot $j\in \{1,\ldots,m\}$,
\[
h\circ^{\id_{\chi(A_j)}}_{(1,j)}\id_{A_j}=h.
\]
For every output slot $i\in \{1,\ldots,n\}$,
\[
\id_{B_i}\circ^{\id_{\chi(B_i)}}_{(i,1)}h=h.
\]
\end{corollary}

\begin{proposition}[Compatibility of CKSC with trace semantics]
\label{prop:cksc_compatible_trace_semantics}
Let $D_1$ and $D_2$ be finite acyclic CKSC diagrams in a CMP $\PC$. Suppose that the $i$-th external output object of $D_1$ is $B$, and that the $j$-th external input object of $D_2$ is $C$. Let $f\in \Int_{\PC}(B,C)$. Let $D$ be the finite acyclic CKSC diagram obtained by connecting the $i$-th external output slot of $D_1$ to the $j$-th external input slot of $D_2$ through the interface color morphism $f$, with external input and output profiles formed by the insertion convention of Definition~\ref{def:cksc}. Then
\[
K_D
=
K_{D_2}\circ^f_{(i,j)}K_{D_1},
\]
where the right-hand side denotes the CKSC integral formula applied to the trace kernels $K_{D_1}$ and $K_{D_2}$ with interface kernel $\kappa_{f;B,C}$.
\end{proposition}

To state associativity and interchange, we use iterated CKSC expressions, in analogy with the uncolored case. A finite acyclic CKSC diagram has a trace kernel by Theorem~\ref{thm:trace_semantics_cksc}. An iterated CKSC expression provides one way of evaluating such a diagram by successive binary CKSC operations. The expression has an evaluation kernel.

\begin{definition}[Iterated CKSC expression]
\label{def:iterated_cksc_expression}
Let $D$ be a finite acyclic CKSC diagram in a CMP $\PC$. An iterated CKSC expression $E$ over $D$ is defined inductively. Each such expression $E$ has an associated subdiagram $D_E$ of $D$ and an evaluation kernel $\mathbf{I}_{D_E}\to \mathbf{O}_{D_E}$.
\begin{enumerate}
\item For every vertex $u\in V(D)$, the single vertex $u$ is an iterated CKSC expression. The associated subdiagram $D_u$ is the one-vertex subdiagram labelled by $k_u:\mathbf{A}_u\to \mathbf{B}_u$. The evaluation kernel of this expression is $k_u$.

\item Let $E_1$ and $E_2$ be iterated CKSC expressions over $D$ whose associated subdiagrams $D_{E_1}$ and $D_{E_2}$ have disjoint vertex sets. Suppose that an internal wire $e=((u,p),(v,q),f_e)$ of $D$ connects an external output slot of $D_{E_1}$ to an external input slot of $D_{E_2}$. Let $r$ be the position of this external output slot in the ordered output profile $\mathbf{O}_{D_{E_1}}$, and let $s$ be the position of this external input slot in the ordered input profile $\mathbf{I}_{D_{E_2}}$. A new iterated CKSC expression $E$ is formed from the data $(E_1,E_2,e)$. Its associated subdiagram $D_E$ is obtained by connecting $D_{E_1}$ to $D_{E_2}$ along the wire $e$, with external profiles formed by the insertion convention of Definition~\ref{def:cksc}. Its evaluation kernel is the CKSC composite $H_2\circ^{f_e}_{(r,s)}H_1$, where $H_1:\mathbf{I}_{D_{E_1}}\to \mathbf{O}_{D_{E_1}}$ and $H_2:\mathbf{I}_{D_{E_2}}\to \mathbf{O}_{D_{E_2}}$ are the evaluation kernels of $E_1$ and $E_2$, respectively.
\end{enumerate}
An iterated CKSC expression $E$ realizes $D$ if $D_E=D$, including the ordered external input and output profiles and the interface color morphisms on internal wires.
\end{definition}

\begin{corollary}[Associativity and interchange for CKSC]
\label{cor:cksc_associativity_interchange}
Let $D$ be a finite acyclic CKSC diagram in a CMP $\PC$. If $E$ is an iterated CKSC expression realizing $D$, then the evaluation kernel of $E$ is the trace kernel $K_D$. Consequently, any two iterated CKSC expressions realizing the same diagram $D$ have the same evaluation kernel. Equivalently, the evaluation of an iterated CKSC expression depends only on the CKSC diagram it realizes.
\end{corollary}

\begin{proof}
We prove the statement by induction on the construction of the iterated CKSC expression $E$.

If $E$ is a single vertex $u$, then $D_E$ is the one-vertex diagram labelled by $k_u$. The evaluation kernel of $E$ is $k_u$. The trace kernel of $D_E$ is also $k_u$. Hence the evaluation kernel of $E$ is $K_{D_E}$.

For the induction step, suppose that $E$ is formed from the data $(E_1,E_2,e)$, where $e=((u,p),(v,q),f_e)$ connects the $r$-th external output slot of $D_{E_1}$ to the $s$-th external input slot of $D_{E_2}$. Let $H_1$ and $H_2$ be the evaluation kernels of $E_1$ and $E_2$, respectively. By definition, the evaluation kernel of $E$ is $H_2\circ^{f_e}_{(r,s)}H_1$. By the induction hypothesis, $H_1=K_{D_{E_1}}$, $H_2=K_{D_{E_2}}$. Therefore, the evaluation kernel of $E$ is $K_{D_{E_2}}\circ^{f_e}_{(r,s)}K_{D_{E_1}}$. By Proposition~\ref{prop:cksc_compatible_trace_semantics}, this kernel is the trace kernel of the subdiagram obtained by connecting $D_{E_1}$ and $D_{E_2}$ along the wire $e$. This subdiagram is $D_E$. Hence, the evaluation kernel of $E$ is $K_{D_E}$.

If $E$ realizes $D$, then $D_E=D$, so the evaluation kernel of $E$ is $K_D$. If $E$ and $E'$ are two iterated CKSC expressions realizing the same diagram $D$, then both evaluation kernels are equal to $K_D$. Hence, the evaluation of an iterated CKSC expression depends only on the CKSC diagram it realizes.
\end{proof}

Corollary~\ref{cor:cksc_associativity_interchange} states that binary CKSC evaluation is independent of the chosen iterated expression realizing a diagram. Associativity and interchange are the special cases in which two different iterated CKSC expressions realize the same finite acyclic CKSC diagram. Associativity corresponds to different bracketings of successive slotwise compositions. Interchange corresponds to different orders of independent slotwise compositions. In both cases, the resulting kernels are equal because both evaluations give the trace kernel of the same diagram.

\begin{example}[A typed diagnosis-treatment workflow]
\label{ex:cksc_diagnosis_treatment}
This toy example illustrates CKSC as a typed probabilistic composition. Let $\PC$ be a CMP with objects $\mathsf{Pat}$, $\mathsf{Bio}$, $\mathsf{Diag}$, and $\mathsf{Treat}$. Their underlying spaces are $\underline{\mathsf{Pat}} = \mathbb{N}$, $\underline{\mathsf{Bio}} = \mathbb{R}$, $\underline{\mathsf{Diag}} = \mathbb{D} = \{\mathsf{bacterial},\mathsf{viral}\}$, $\underline{\mathsf{Treat}} = \mathbb{T} = \{\mathsf{antibiotic},\mathsf{supportive}\}$. The finite sets are equipped with their power-set $\sigma$-algebras, and $\mathbb{R}$ is equipped with its Borel $\sigma$-algebra.

Let $k:(\mathsf{Pat})\to(\mathsf{Bio})$ be a biomarker kernel. For simplicity, assume that the biomarker value is a standardized real-valued score $z$ whose conditional law does not depend on the patient index $p\in \mathbb{N}$. Thus, $k(dz\mid p)=\phi(z)\,dz$, where $\phi$ is the standard normal density.

Let $c_{\mathsf{bio}}=\chi(\mathsf{Bio})$ and $c_{\mathsf{diag}}=\chi(\mathsf{Diag})$. Suppose that there is an interface color morphism $f\in \Int_{\PC}(\mathsf{Bio},\mathsf{Diag})$ with associated interface kernel $\kappa_{f;\mathsf{Bio},\mathsf{Diag}}:\mathsf{Bio}\to \mathsf{Diag}$. Define this interface kernel by
\[
\kappa_{f;\mathsf{Bio},\mathsf{Diag}}(\{d\}\mid z)
=
\begin{cases}
\sigma(z), & d=\mathsf{bacterial},\\
1-\sigma(z), & d=\mathsf{viral},
\end{cases}
\]
for $d\in \mathbb{D}$ and $z\in \mathbb{R}$, where $\sigma(z)=1/(1+\exp(-z))$.

Finally, let $l:(\mathsf{Diag})\to(\mathsf{Treat})$ be the deterministic treatment kernel $l(d\tau\mid d)=\delta_{r(d)}(d\tau)$, where $r(\mathsf{bacterial})=\mathsf{antibiotic}$ and $r(\mathsf{viral})=\mathsf{supportive}$.

The CKSC composite $h=l\circ^f_{(1,1)}k$ has signature $h:(\mathsf{Pat})\to(\mathsf{Treat})$. For every $E\subseteq \mathbb{T}$ and $p\in \mathbb{N}$, Definition~\ref{def:cksc} gives
\[
h(E\mid p)
=
\int_{\mathbb{R}}
\sum_{d\in \mathbb{D}}
\sum_{\tau\in \mathbb{T}}
\mathbf{1}_E(\tau)\,
l(\{\tau\}\mid d)\,
\kappa_{f;\mathsf{Bio},\mathsf{Diag}}(\{d\} \mid z)\,
\phi(z)\,dz.
\]
Since $l$ is deterministic, this reduces to
\[
h(E\mid p)
=
\int_{\mathbb{R}}
\sum_{d\in \mathbb{D}}
\mathbf{1}_E(r(d))\,
\kappa_{f;\mathsf{Bio},\mathsf{Diag}}(\{d\} \mid z)\,
\phi(z)\,dz.
\]
In particular,
\[
h(\{\mathsf{antibiotic}\}\mid p)
=
\int_{\mathbb{R}}\sigma(z)\phi(z)\,dz
=
\frac{1}{2}.
\]
The last equality follows from the symmetry of the standard normal density and the identity $\sigma(z)+\sigma(-z)=1$.

The example shows how CKSC separates two roles. The morphism $k$ produces a real-valued biomarker. The interface kernel determined by $f$ converts this value into a distribution on a different typed object, namely diagnoses. The morphism $l$ then maps diagnoses to treatments. The intermediate biomarker and diagnosis variables are integrated out in the composite kernel $h$.
\end{example}

\section{Co-indexed colored Markov polycategories}
\label{sec:ccmp}

This section introduces co-indexed colored Markov polycategories. The purpose of co-indexing is to organize a family of CMPs over an indexing category $\T$. Each object $t\in \Ob(\T)$ carries a CMP $\PC_t$, and each morphism $\alpha:t\to t'$ gives a structure-preserving pushforward from $\PC_t$ to $\PC_{t'}$. Proofs not included in the main text of this section are given in Appendix~\ref{app:functorial_proofs}.

\subsection{Definition of co-indexed colored Markov polycategory}
\label{subsec:def_ccmp}

We first prove that CMPs over a fixed color system form a category. This is the target category for the co-indexed state functor.

\begin{proposition}[Category of CMPs]
\label{prop:category_of_cmps}
Fix a color system $(\K,\M,\iota)$. CMPs over $(\K,\M,\iota)$ and CMP-functors between them form a category. This category is denoted by $\mathbf{CMP}_{\K,\M,\iota}$.
\end{proposition}

We now define the dynamic structure used in the rest of the paper. A co-indexed colored Markov polycategory consists of a state functor, which transports CMPs, and a parameter functor, which transports parameter spaces.

\begin{definition}[Co-indexed colored Markov polycategory]
\label{def:ccmp}
Fix a color system $(\K,\M,\iota)$, and let $\mathbf{CMP}_{\K,\M,\iota}$ be the category of CMPs over this color system. Let $\T$ be an indexing category. Let $\ThetaC$ be the category whose objects are finite-dimensional real vector spaces and whose morphisms are differentiable maps.

A co-indexed colored Markov polycategory (CCMP) over $\T$ consists of two strict functors $F_{\PC}:\T\to \mathbf{CMP}_{\K,\M,\iota},~F_\theta:\T\to \ThetaC$.

For each object $t\in \Ob(\T)$, write $\PC_t=F_{\PC}(t),~\Theta_t=F_\theta(t)$. The CMP $\PC_t$ is called the state CMP at $t$, and $\Theta_t$ is called the parameter space at $t$.

For each morphism $\alpha:t\to t'$ in $\T$, write $\alpha_{!,s}=F_{\PC}(\alpha):\PC_t\to \PC_{t'}$ and $\alpha_{!,\theta}=F_\theta(\alpha):\Theta_t\to \Theta_{t'}$. The functor $\alpha_{!,s}$ is called the state pushforward along $\alpha$, and the differentiable map $\alpha_{!,\theta}$ is called the parameter pushforward along $\alpha$.

Strict functoriality means that, for every $t\in \Ob(\T)$, $(\id_t)_{!,s}=\Id_{\PC_t},~(\id_t)_{!,\theta}=\id_{\Theta_t}$, and, for every composable pair $t\xrightarrow{\alpha}t'\xrightarrow{\beta}t''$, $(\beta\circ \alpha)_{!,s} = \beta_{!,s}\circ \alpha_{!,s}$ and $(\beta\circ \alpha)_{!,\theta} = \beta_{!,\theta}\circ \alpha_{!,\theta}$.
\end{definition}

The state functor $F_{\PC}$ specifies how typed stochastic systems are transported along transitions in $\T$. Since its morphisms are CMP-functors, state pushforwards preserve object colors, morphism colors, interface kernels, and KSC. The parameter functor $F_\theta$ records how parameter spaces transform along the same transitions. The two functors are kept separate because the categorical structure of a state and the differentiable structure of its parameter space play different roles.

\subsection{Preservation of diagram semantics under CMP-functors}
\label{subsec:cmp_functors_preserve_semantics}

CMP-functors preserve typed diagrams and their binary CKSC evaluations. We state this for diagrams that are realized by iterated CKSC expressions (Definition~\ref{def:iterated_cksc_expression}). This is the case in which the trace kernel of the diagram is represented as a morphism obtained by successive CKSC operations.

Let $G:\PC\to \PC'$ be a CMP-functor. If $D$ is a finite acyclic CKSC diagram in $\PC$, its image $G D$ is the finite acyclic CKSC diagram in $\PC'$ obtained by applying $G$ to every vertex label and leaving every interface color morphism unchanged. Thus, if a vertex of $D$ is labelled by $k_u:\mathbf{A}_u\to \mathbf{B}_u$, then the corresponding vertex of $G D$ is labelled by $Gk_u:G\mathbf{A}_u\to G\mathbf{B}_u$. If an internal wire of $D$ is labelled by $f\in \Int_{\PC}(B,C)$, then the corresponding internal wire of $G D$ is labelled by the same interface color morphism $f\in \Int_{\PC'}(GB,GC)$, which is well-defined because $G$ preserves interface systems.

A finite acyclic CKSC diagram is called CKSC-reducible if it is realized by an iterated CKSC expression. For such a diagram $D$, Corollary~\ref{cor:cksc_associativity_interchange} shows that every iterated CKSC expression realizing $D$ has the same evaluation kernel, namely the trace kernel $K_D$. When this kernel is regarded as the morphism obtained by successive CKSC operations, we denote it by $\langle D\rangle_{\PC}$.

\begin{proposition}[CMP-functors preserve CKSC reductions]
\label{prop:cmp_functors_preserve_cksc_reductions}
Let $G:\PC\to \PC'$ be a CMP-functor. If $D$ is a CKSC-reducible finite acyclic CKSC diagram in $\PC$, then $GD$ is CKSC-reducible in $\PC'$, and $G\langle D\rangle_{\PC} = \langle GD\rangle_{\PC'}$. Consequently, the Markov kernel underlying $G\langle D\rangle_{\PC}$ is the trace kernel of $GD$.
\end{proposition}

\subsection{State pushforwards and parameter pushforwards}
\label{subsec:state_parameter_pushforwards}

Let $(F_{\PC},F_\theta)$ be a CCMP over $\T$. A morphism $\alpha:t\to t'$ in $\T$ acts on the state CMP by the CMP-functor $\alpha_{!,s}:\PC_t\to \PC_{t'}$. Thus, every object $X\in \Ob(\PC_t)$ has an image $\alpha_{!,s}X\in \Ob(\PC_{t'})$, and every morphism $k:\mathbf{A}\to \mathbf{B}$ in $\PC_t$ has an image $\alpha_{!,s}k:\alpha_{!,s}\mathbf{A}\to \alpha_{!,s}\mathbf{B}$ in $\PC_{t'}$.

If $D$ is a finite acyclic CKSC diagram in $\PC_t$, its state pushforward along $\alpha$ is the finite acyclic CKSC diagram $\alpha_{!,s}D$ in $\PC_{t'}$ obtained by applying $\alpha_{!,s}$ to every object and vertex kernel of $D$, while preserving the interface color morphisms as required by Definition~\ref{def:cmp_functor}. The ordered external input and output profiles of $\alpha_{!,s}D$ are the images of the ordered external profiles of $D$.

\begin{proposition}[Functoriality of state pushforwards on diagrams]
\label{prop:state_pushforward_diagrams_functorial}
Let $(F_{\PC},F_\theta)$ be a CCMP over $\T$. Let $D$ be a finite acyclic CKSC diagram in $\PC_t$. First, $(\id_t)_{!,s}D=D$. Second, for every composable pair $t\xrightarrow{\alpha}t'\xrightarrow{\beta}t''$, $(\beta\circ\alpha)_{!,s}D=\beta_{!,s}(\alpha_{!,s}D)$. If $D$ is CKSC-reducible, then $\alpha_{!,s}\langle D\rangle_{\PC_t} = \langle \alpha_{!,s}D\rangle_{\PC_{t'}}$.
\end{proposition}

\begin{proof}
The first two equations follow directly from strict functoriality of $F_{\PC}$. The identity morphism of $\T$ is sent to the identity CMP-functor, and a composite $\beta\circ\alpha$ is sent to the composite CMP-functor $\beta_{!,s}\circ \alpha_{!,s}$. Applying these equalities to every object, vertex kernel, and interface color morphism of $D$ gives $(\id_t)_{!,s}D=D$ and, for every composable pair $t\xrightarrow{\alpha}t'\xrightarrow{\beta}t''$, $(\beta\circ\alpha)_{!,s}D=\beta_{!,s}(\alpha_{!,s}D)$.

If $D$ is CKSC-reducible, then $\alpha_{!,s}\langle D\rangle_{\PC_t} = \langle \alpha_{!,s}D\rangle_{\PC_{t'}}$ follows from Proposition~\ref{prop:cmp_functors_preserve_cksc_reductions}, applied to the CMP-functor $\alpha_{!,s}:\PC_t\to \PC_{t'}$.
\end{proof}

The parameter pushforward along the same transition is the differentiable map $\alpha_{!,\theta}:\Theta_t\to \Theta_{t'}$. It is independent of the state pushforward as a map, but it is indexed by the same transition $\alpha$. Strict functoriality gives $(\id_t)_{!,\theta}=\id_{\Theta_t}$ and $(\beta\circ\alpha)_{!,\theta}=\beta_{!,\theta}\circ \alpha_{!,\theta}$ for every composable pair $t\xrightarrow{\alpha}t'\xrightarrow{\beta}t''$. Hence, parameters are transported consistently along multi-step transitions in $\T$.

\begin{example}[Dynamic graph systems as CCMPs]
\label{ex:dynamic_graph_ccmp}
A simple class of examples is given by dynamic graph systems. To avoid irrelevant deletion issues, consider first a graph process whose transitions are graph inclusions. Thus, the indexing category $\T$ has objects finite graphs $G=(V,E)$, and a morphism $\alpha:G\to G'$ is an inclusion of $G$ as a subgraph of $G'$.

Fix object colors $c_V,~c_E,~c_G$ for vertex features, edge features, and global graph-level features. For a graph $G=(V,E)$, let $\PC_G$ be a CMP whose objects include a vertex-feature object $X_v$ of color $c_V$ for each $v\in V$, an edge-feature object $X_e$ of color $c_E$ for each $e\in E$, and possibly a global-context object $X_G$ of color $c_G$. Typical underlying spaces are Euclidean spaces such as $\mathbb{R}^{d_V}$, $\mathbb{R}^{d_E}$, and $\mathbb{R}^{d_G}$ with their Borel $\sigma$-algebras.

Morphisms in $\PC_G$ may represent local stochastic update kernels. For example, if $e=(u,v)$ is an edge, one possibility is to include a kernel $k_e:(X_u,X_e,X_v)\to (X'_e)$ which updates the edge feature from the features of its incident vertices and its current edge feature. Similarly, a vertex update kernel may have the form $k_v:(X_v,\mathbf{X}_{N(v)})\to (X'_v)$, where $\mathbf{X}_{N(v)}$ is an ordered tuple of feature objects associated with the neighbors of $v$. The ordering is part of the chosen model structure.

For an inclusion $\alpha:G\to G'$, the state pushforward $\alpha_{!,s}:\PC_G\to \PC_{G'}$ maps the object $X_v$ of a persistent vertex $v\in V$ to the corresponding object in $\PC_{G'}$, and similarly for persistent edge objects. Kernels whose local incidence pattern is preserved are mapped to the corresponding kernels in $\PC_{G'}$. New vertices and edges in $G'$ are not images of old objects. They belong to the larger target CMP $\PC_{G'}$.

A parameter space $\Theta_G$ may contain the parameters of the update kernels used in $\PC_G$. The parameter pushforward $\alpha_{!,\theta}:\Theta_G\to \Theta_{G'}$ may copy shared parameters for persistent kernel types and initialize parameters for kernels associated with new vertices or edges. The functoriality equations in Definition~\ref{def:ccmp} express that performing two graph inclusions in sequence transports state structure and parameters in the same way as their composite inclusion.

This example is deliberately schematic. Its role is to show how the CCMP formalism separates the typed stochastic structure at each graph state from the functorial transport of that structure along graph transitions.
\end{example}

\section{Diagrammatic differentiation}
\label{sec:diagrammatic_differentiation}

This section develops a bounded differentiation result for finite acyclic CKSC diagrams. The shape of the diagram and its interface color morphisms are fixed. Parameters vary only the vertex kernels. Interface kernels are treated as fixed structure. A learned interface kernel can instead be represented as a parameterized vertex. Proofs not included in the main text of this section are given in Appendix~\ref{app:differentiation_proofs}.

The goal is to define an expected scalar objective for a parameterized CKSC diagram and to state conditions under which its gradient can be computed by a reverse traversal of the diagram.

\subsection{Parameterized CKSC diagrams}
\label{subsec:parameterized_cksc_diagrams}

Let $\PC$ be a CMP. A parameterized kernel is a family of kernels with a fixed source profile, target profile, and morphism color.

\begin{definition}[Parameterized kernel]
\label{def:parameterized_kernel}
Let $\Theta$ be a finite-dimensional real vector space. Let $\mathbf{A}$ and $\mathbf{B}$ be finite ordered tuples of objects of $\PC$. A parameterized kernel of type $\mathbf{A}\to \mathbf{B}$ with parameter space $\Theta$ is a family $\{k_\theta:\mathbf{A}\to \mathbf{B}\}_{\theta\in \Theta}$ such that each $k_\theta$ is a morphism of the underlying Markov polycategory of $\PC$.

The family is required to be jointly measurable, that is, for every $E\in \Sigma_{\underline{\mathbf{B}}}$, the map $(\theta,\mathbf{a})\mapsto k_\theta(E\mid \mathbf{a})$ from $\Theta\times \underline{\mathbf{A}}$ to $[0,1]$ is measurable.

The family has constant morphism color if there is a morphism color $m\in \Hom_{\M}(\chi(\mathbf{A});\chi(\mathbf{B}))$ such that $\psi(k_\theta)=m$ for every $\theta\in \Theta$.
\end{definition}

In the rest of this section, parameterized kernels are assumed to have constant morphism color. This keeps the typed structure of a diagram fixed while the numerical kernels vary.

\begin{definition}[Parameterized CKSC diagram]
\label{def:parameterized_cksc_diagram}
A parameterized CKSC diagram $D$ in $\PC$ consists of the following data.
\begin{enumerate}
\item There is a finite acyclic CKSC diagram shape. Thus, $D$ has a finite vertex set $V(D)$, fixed input and output profiles at each vertex, fixed internal wires, fixed interface color morphisms on internal wires, and fixed ordered external input and output profiles $\mathbf{I}_D$ and $\mathbf{O}_D$.
\item For every vertex $u\in V(D)$, there is a finite-dimensional real vector space $\Theta_u$ and a parameterized kernel $\{k_{u,\theta_u}:\mathbf{A}_u\to \mathbf{B}_u\}_{\theta_u\in \Theta_u}$.
\end{enumerate}
The total parameter space of $D$ is $\Theta_D=\prod_{u\in V(D)}\Theta_u$. For $\theta=(\theta_u)_{u\in V(D)}\in \Theta_D$, the instantiated diagram $D_\theta$ is the finite acyclic CKSC diagram obtained by labelling each vertex $u$ by the kernel $k_{u,\theta_u}$ and keeping all internal interface color morphisms fixed.
\end{definition}

A vertex $u\in V(D)$ is called parameterized if $\Theta_u$ is not the zero vector space. A vertex whose parameter space is the zero space carries a fixed kernel and contributes a trivial factor to $\Theta_D$.

By Theorem~\ref{thm:trace_semantics_cksc}, every instantiated diagram $D_\theta$ determines a Markov kernel $K_D^\theta:\mathbf{I}_D\to \mathbf{O}_D$. The superscript $\theta$ indicates that the trace kernel depends on the vertex parameters. The ordered diagram shape and its interface color morphisms are not varied.

\subsection{Objective kernels and expected objectives}
\label{subsec:objective_kernels_expected_objectives}

To differentiate a parameterized diagram, its output is turned into a scalar random variable. We do this by adding a deterministic objective kernel. This covers the supervised case, where the diagram output is a prediction and the reference variable is a label.

\begin{definition}[Objective kernel]
\label{def:objective_kernel}
Let $V$ be an object with underlying measurable space $(\mathbb{R},\mathcal{B}(\mathbb{R}))$. Let $\mathbf{O}$ and $\mathbf{R}$ be finite ordered tuples of objects of a CMP $\PC$. An objective function is a measurable map $f_O:\underline{\mathbf{O}}\times \underline{\mathbf{R}}\to \mathbb{R}$. The corresponding objective kernel is the deterministic Markov kernel $k_O:(\mathbf{O},\mathbf{R})\to (V)$ defined by $k_O(dv\mid \mathbf{o},\mathbf{r})=\delta_{f_O(\mathbf{o},\mathbf{r})}(dv)$.
\end{definition}

The tuple $\mathbf{O}$ is the external output profile of a parameterized diagram. The tuple $\mathbf{R}$ is a reference profile. It may contain labels, targets, or other variables relative to which the diagram output is evaluated. If no reference variables are needed, $\mathbf{R}$ is taken to be the empty tuple.

\begin{definition}[Expected objective of a parameterized diagram]
\label{def:expected_objective}
Let $D$ be a parameterized CKSC diagram in a CMP $\PC$, with external input profile $\mathbf{I}_D$ and external output profile $\mathbf{O}_D$. Let $\mathbf{R}$ be a finite ordered tuple of reference objects. Let $\rho$ be a probability measure on $\underline{\mathbf{I}_D}\times \underline{\mathbf{R}}$. Let $f_O:\underline{\mathbf{O}_D}\times \underline{\mathbf{R}}\to \mathbb{R}$ be an objective function.

For $\theta\in \Theta_D$, define the scalar random objective $J$ by the following sampling procedure:
\begin{align*}
(\mathbf{X},\mathbf{R}) &\sim \rho,\\
\mathbf{Y} &\sim K_D^\theta(\,\cdot \mid \mathbf{X}),\\
J &= f_O(\mathbf{Y},\mathbf{R}).
\end{align*}
Here, $\mathbf{X}$ is the external input random tuple, $\mathbf{R}$ is the reference random tuple, and $\mathbf{Y}$ is the external output random tuple of the diagram.

The expected objective of $D$ at parameter value $\theta$ is
\[
\mathcal{L}_D(\theta)
=
\mathbb{E}[J]
=
\int_{\underline{\mathbf{I}_D}\times \underline{\mathbf{R}}}
\left(
\int_{\underline{\mathbf{O}_D}}
f_O(\mathbf{y},\mathbf{r})\,
K_D^\theta(d\mathbf{y}\mid \mathbf{x})
\right)
\rho(d\mathbf{x},d\mathbf{r}).
\]
\end{definition}

The probability measure $\rho$ is not part of the CMP structure. It is external data for the learning problem. In supervised learning, $\rho$ is the data-generating distribution on inputs and labels. In a deterministic objective such as squared error or cross-entropy, $f_O$ is the loss function and $\mathcal{L}_D$ is the expected loss.

The notation $J$ is reserved for the scalar random objective. A realized scalar value is denoted by $v$ or $\ell$. The function $\mathcal{L}_D:\Theta_D\to \mathbb{R}$ is the quantity whose gradient is studied.

\subsection{Conditional expected future objectives}
\label{subsec:conditional_expected_future_objectives}

The reverse rules below use conditional expectations of the scalar objective given selected trace variables. We define these quantities with respect to the joint trace law of an instantiated parameterized diagram.

Let $D$ be a parameterized CKSC diagram, let $\theta\in \Theta_D$, and let $D_\theta$ be the corresponding instantiated CKSC diagram. Let $\widehat{D}_\theta$ be its interface expansion. Write $\widehat{V}(D)$ for the vertex set of the interface expansion. This set contains the original vertices of $D$ and the additional unary interface vertices.

Choose a topological ordering of $\widehat{V}(D)$. By Definition~\ref{def:trace_measure_ksc}, this gives a conditional trace measure on the product of the output spaces of all vertices of $\widehat{D}_\theta$. Since the variables are labelled by vertices, we regard this measure as a measure on the labelled product $\prod_{u\in \widehat{V}(D)} \underline{\mathbf{B}_u}$. The adjacent-swap argument in the proof of Theorem~\ref{thm:trace_semantics_ksc} shows that this labelled joint law is independent of the chosen topological ordering.

Let $\rho$ be the probability measure on $\underline{\mathbf{I}_D}\times \underline{\mathbf{R}}$ from Definition~\ref{def:expected_objective}. The full trace law of the expected-objective problem at parameter value $\theta$ is the probability measure $\nu_D^\theta$ on $\underline{\mathbf{I}_D}\times \underline{\mathbf{R}}\times \prod_{u\in \widehat{V}(D)} \underline{\mathbf{B}_u}$ defined by
\[
\nu_D^\theta(d\mathbf{x},d\mathbf{r},d(\mathbf{y}_u)_{u\in \widehat{V}(D)})
=
\rho(d\mathbf{x},d\mathbf{r})\,
\mu_{\widehat{D}}^\theta(d(\mathbf{y}_u)_{u\in \widehat{V}(D)}\mid \mathbf{x}),
\]
where $\mu_{\widehat{D}}^\theta(\,\cdot\mid \mathbf{x})$ is the trace measure of the expanded KSC diagram $\widehat{D}_\theta$ conditional on the external input value $\mathbf{x}$.

Random variables under $\nu_D^\theta$ are denoted with overlines. Thus, $\overline{\mathbf{X}}$ is the external input tuple, $\overline{\mathbf{R}}$ is the reference tuple, and $\overline{\mathbf{Y}}_u$ is the output tuple of the vertex $u\in \widehat{V}(D)$. The external output tuple of the whole diagram is denoted by $\overline{\mathbf{Y}}_D$. It is obtained from the vertex output variables by the external output projection map. The scalar objective is $J=f_O(\overline{\mathbf{Y}}_D,\overline{\mathbf{R}})$.

\begin{definition}[Conditional expected future objective]
\label{def:conditional_expected_future_objective}
Let $\overline{\mathbf{Z}}$ be a finite tuple of random variables under the trace law $\nu_D^\theta$, taking values in a standard Borel space $\underline{\mathbf{Z}}$. A conditional expected future objective for $\overline{\mathbf{Z}}$ is a measurable function $Q_{\overline{\mathbf{Z}}}^\theta:\underline{\mathbf{Z}}\to \mathbb{R}$ such that
\[
Q_{\overline{\mathbf{Z}}}^\theta(\overline{\mathbf{Z}})
=
\mathbb{E}_{\nu_D^\theta}[J\mid \sigma(\overline{\mathbf{Z}})].
\]
\end{definition}

Since all spaces in the paper are standard Borel, such measurable versions exist for integrable $J$. They are unique only up to almost-sure equality with respect to the law of $\overline{\mathbf{Z}}$. This is sufficient for gradient identities, which are stated as identities of expectations.

The notation $Q_{\overline{\mathbf{Z}}}^\theta(\mathbf{z})$ denotes evaluation of the chosen version at $\mathbf{z}$. We avoid notation such as $\mathbb{E}[J\mid \mathbf{z}]$. When $\overline{\mathbf{Z}}=(\overline{\mathbf{Y}}_u,\overline{\mathbf{C}})$ consists of the output tuple of a vertex together with a context tuple, the function $Q_{\overline{\mathbf{Z}}}^\theta$ is the objective seen from that vertex.

\subsection{Local parameter-gradient representing functions}
\label{subsec:admissible_local_gradient_operators}

The gradient theory herein is stated in terms of local parameter-gradient representing functions. This separates the diagrammatic part of the argument from the analytic method used to differentiate a kernel. Different kernels may use different local rules, such as score-function rules or deterministic pathwise rules.

Let $\Theta$ be a finite-dimensional real vector space. We write $\Theta^*$ for its dual. After choosing coordinates, elements of $\Theta^*$ may be identified with gradient vectors.

\begin{definition}[Local parameter-gradient representing function]
\label{def:admissible_local_parameter_gradient_operator}
Let $\{k_\theta:\mathbf{A}\to \mathbf{B}\}_{\theta\in \Theta}$ be a parameterized kernel in a CMP $\PC$. Let $\mathbf{S}$ be a finite ordered tuple of context objects, and let $Q:\underline{\mathbf{A}}\times \underline{\mathbf{B}}\times \underline{\mathbf{S}}\to \mathbb{R}$ be a measurable function. For every probability measure $\lambda$ on $\underline{\mathbf{A}}\times \underline{\mathbf{S}}$ which does not depend on $\theta$, define
\[
F_{Q,\lambda}(\theta)
=
\int_{\underline{\mathbf{A}}\times \underline{\mathbf{S}}}
\left(
\int_{\underline{\mathbf{B}}}
Q(\mathbf{a},\mathbf{b},\mathbf{s})\,
k_\theta(d\mathbf{b}\mid \mathbf{a})
\right)
\lambda(d\mathbf{a},d\mathbf{s}).
\]
A local parameter-gradient representing function for $k_\theta$ relative to $\mathbf{S}$ and $Q$ is a measurable function $G_{k,Q}: \Theta\times \underline{\mathbf{A}}\times \underline{\mathbf{B}}\times \underline{\mathbf{S}} \to \Theta^*$ such that for every $\theta\in \Theta$ at which $F_{Q,\lambda}$ is finite and differentiable, the derivative satisfies
\[
D_\theta F_{Q,\lambda}(\theta)
=
\int_{\underline{\mathbf{A}}\times \underline{\mathbf{S}}}
\left(
\int_{\underline{\mathbf{B}}}
G_{k,Q}(\theta,\mathbf{a},\mathbf{b},\mathbf{s})\,
k_\theta(d\mathbf{b}\mid \mathbf{a})
\right)
\lambda(d\mathbf{a},d\mathbf{s}).
\]
\end{definition}

The representing function $G_{k,Q}$ represents the derivative of the local expected objective $F_{Q,\lambda}$ as an expectation over the local trace variables. If such a function is specified for every $Q$ in a class $\mathcal{Q}$, we say that the parameterized kernel has local parameter-gradient representing functions for $\mathcal{Q}$.

The tuple $\mathbf{S}$ represents variables that are not outputs of the local kernel but are needed to evaluate the downstream objective. In a CKSC diagram, these variables may include passthrough wires, reference variables, or other trace variables. The measure $\lambda$ represents the joint law of the local input and context variables induced by the part of the diagram upstream of the local kernel.

Definition~\ref{def:admissible_local_parameter_gradient_operator} is local. It does not require the whole diagram to be differentiable at once. It requires that, once the local input and context distribution is fixed, the derivative of the local expected objective can be represented as the expectation of the trace-wise quantity $G_{k,Q}$.

\begin{proposition}[Score-function local representing function]
\label{prop:score_function_local_operator}
Let $\{k_\theta:\mathbf{A}\to \mathbf{B}\}_{\theta\in \Theta}$ be a parameterized kernel. Suppose that there is a $\sigma$-finite measure $\nu_{\mathbf{B}}$ on $\underline{\mathbf{B}}$ such that $k_\theta(d\mathbf{b}\mid \mathbf{a})=p_\theta(\mathbf{b}\mid \mathbf{a})\,\nu_{\mathbf{B}}(d\mathbf{b})$ for all $\theta\in \Theta$ and $\mathbf{a}\in \underline{\mathbf{A}}$. Let $Q:\underline{\mathbf{A}}\times \underline{\mathbf{B}}\times \underline{\mathbf{S}}\to \mathbb{R}$ be measurable. For a probability measure $\lambda$ on $\underline{\mathbf{A}}\times \underline{\mathbf{S}}$ considered in Definition~\ref{def:admissible_local_parameter_gradient_operator}, write $\lambda_{\mathbf{A}}$ for its marginal on $\underline{\mathbf{A}}$. Assume that, for every such $\lambda$, the density satisfies $p_\theta(\mathbf{b}\mid \mathbf{a})>0$ for $(\lambda_{\mathbf{A}}\otimes \nu_{\mathbf{B}})$-almost every $(\mathbf{a},\mathbf{b})$, and that $\theta\mapsto p_\theta(\mathbf{b}\mid \mathbf{a})$ is differentiable for such $(\mathbf{a},\mathbf{b})$. Assume also that, for every such $\lambda$, differentiation may be passed under the integral:
\[
D_\theta
\int
Q(\mathbf{a},\mathbf{b},\mathbf{s})
p_\theta(\mathbf{b}\mid \mathbf{a})
\,\nu_{\mathbf{B}}(d\mathbf{b})
\,\lambda(d\mathbf{a},d\mathbf{s})
=
\int
Q(\mathbf{a},\mathbf{b},\mathbf{s})
D_\theta p_\theta(\mathbf{b}\mid \mathbf{a})
\,\nu_{\mathbf{B}}(d\mathbf{b})
\,\lambda(d\mathbf{a},d\mathbf{s}).
\]
Then the function
\[
G_{k,Q}^{\mathrm{sf}}(\theta,\mathbf{a},\mathbf{b},\mathbf{s})
=
Q(\mathbf{a},\mathbf{b},\mathbf{s})\,
D_\theta \log p_\theta(\mathbf{b}\mid \mathbf{a}),
\]
defined on the set where $p_\theta(\mathbf{b}\mid \mathbf{a})>0$ and extended by zero elsewhere, is a local parameter-gradient representing function.
\end{proposition}

\begin{proposition}[Deterministic pathwise local representing function]
\label{prop:deterministic_pathwise_local_operator}
Let $\{k_\theta:\mathbf{A}\to \mathbf{B}\}_{\theta\in \Theta}$ be a parameterized kernel. Suppose that $\underline{\mathbf{B}}$ is a finite-dimensional real vector space with its Borel $\sigma$-algebra, and that there is a measurable map $g:\Theta\times \underline{\mathbf{A}}\to \underline{\mathbf{B}}$ such that $k_\theta(d\mathbf{b}\mid \mathbf{a})=\delta_{g(\theta,\mathbf{a})}(d\mathbf{b})$. Let $Q:\underline{\mathbf{A}}\times \underline{\mathbf{B}}\times \underline{\mathbf{S}}\to \mathbb{R}$ be measurable. Assume that, for every $(\mathbf{a},\mathbf{s})$, the map $\mathbf{b}\mapsto Q(\mathbf{a},\mathbf{b},\mathbf{s})$ is differentiable. Assume also that, for every $\mathbf{a}$, the map $\theta\mapsto g(\theta,\mathbf{a})$ is differentiable. For every probability measure $\lambda$ on $\underline{\mathbf{A}}\times \underline{\mathbf{S}}$ considered in Definition~\ref{def:admissible_local_parameter_gradient_operator}, assume that the function
\[
F_Q(\theta)
=
\int_{\underline{\mathbf{A}}\times \underline{\mathbf{S}}}
Q(\mathbf{a},g(\theta,\mathbf{a}),\mathbf{s})\,
\lambda(d\mathbf{a},d\mathbf{s})
\]
is differentiable and that differentiation may be passed under the integral. Then
\[
G_{k,Q}^{\mathrm{det}}(\theta,\mathbf{a},\mathbf{b},\mathbf{s})
=
D_{\mathbf{b}}Q(\mathbf{a},\mathbf{b},\mathbf{s})
\circ
D_\theta g(\theta,\mathbf{a})
\]
is a local parameter-gradient representing function, where $D_{\mathbf{b}}Q(\mathbf{a},\mathbf{b},\mathbf{s})$ is a covector on $\underline{\mathbf{B}}$, and $D_\theta g(\theta,\mathbf{a})$ is a linear map from $\Theta$ to $\underline{\mathbf{B}}$.
\end{proposition}

\subsection{The local reverse rule for CKSC}
\label{subsec:local_reverse_rule_cksc}

We now state the local reverse rule. The rule isolates the contribution of one parameterized vertex while all other vertex parameters are held fixed.

\begin{definition}[Local context]
\label{def:admissible_local_context}
Let $D$ be a parameterized CKSC diagram, and let $u\in V(D)$ be a parameterized vertex. Write $k_{u,\theta_u}:\mathbf{A}_u\to \mathbf{B}_u$ for its parameterized kernel. Under the full trace law $\nu_D^\theta$, let $\overline{\mathbf{A}}_u$ be the random input tuple of the vertex $u$, and let $\overline{\mathbf{B}}_u$ be its random output tuple. Let $u\in V(D)$ be a parameterized vertex. A finite tuple of trace random variables $\overline{\mathbf{S}}_u$ is a local context for $u$ if the following condition holds. For every fixed value of the parameters $\theta_{-u}$ at all vertices except $u$, there is a probability measure $\lambda_u^{\theta_{-u}}$ on $\underline{\mathbf{A}_u}\times \underline{\mathbf{S}_u}$ such that $\lambda_u^{\theta_{-u}}$ does not depend on $\theta_u$, and for every bounded measurable function $\varphi:\underline{\mathbf{A}_u}\times \underline{\mathbf{B}_u}\times \underline{\mathbf{S}_u}\to \mathbb{R}$,
\[
\mathbb{E}_{\nu_D^\theta}
[
\varphi(\overline{\mathbf{A}}_u,\overline{\mathbf{B}}_u,\overline{\mathbf{S}}_u)
]
=
\int_{\underline{\mathbf{A}_u}\times \underline{\mathbf{S}_u}}
\left(
\int_{\underline{\mathbf{B}_u}}
\varphi(\mathbf{a},\mathbf{b},\mathbf{s})\,
k_{u,\theta_u}(d\mathbf{b}\mid \mathbf{a})
\right)
\lambda_u^{\theta_{-u}}(d\mathbf{a},d\mathbf{s}).
\]
\end{definition}

\begin{lemma}[Extension of the local context identity]
\label{lem:local_context_extension}
Let $\overline{\mathbf{S}}_u$ be a local context for $u$, with associated probability measure $\lambda_u^{\theta_{-u}}$. Then the identity in Definition~\ref{def:admissible_local_context} also holds for every measurable function $\varphi:\underline{\mathbf{A}_u}\times \underline{\mathbf{B}_u}\times \underline{\mathbf{S}_u}\to \mathbb{R}$ that is integrable with respect to the joint law of $(\overline{\mathbf{A}}_u,\overline{\mathbf{B}}_u,\overline{\mathbf{S}}_u)$ under $\nu_D^\theta$, and both sides are then finite. Moreover, the identity holds component-wise for every measurable function $\Phi:\underline{\mathbf{A}_u}\times \underline{\mathbf{B}_u}\times \underline{\mathbf{S}_u}\to \Theta_u^*$ whose components, relative to any chosen basis of $\Theta_u^*$, are integrable in this way. The componentwise identity does not depend on the chosen basis.
\end{lemma}

The condition in Definition~\ref{def:admissible_local_context} says that, after the local input and context have been fixed, the dependence on $\theta_u$ enters through the local kernel $k_{u,\theta_u}$. Proposition~\ref{prop:local_reverse_rule} also assumes a version $Q_u^{\theta_{-u}}$ of the conditional expected future objective that is fixed with respect to $\theta_u$ when the other parameters $\theta_{-u}$ are held fixed. This assumption is used in the proof to ensure that differentiation with respect to $\theta_u$ acts on the local kernel rather than on the objective function.

\begin{proposition}[Local reverse rule]
\label{prop:local_reverse_rule}
Let $D$ be a parameterized CKSC diagram with expected objective $\mathcal{L}_D(\theta)=\mathbb{E}_{\nu_D^\theta}[J]$. Fix a parameterized vertex $u\in V(D)$, and write $\theta=(\theta_u,\theta_{-u})$, where $\theta_{-u}$ denotes the parameters at all vertices except $u$. Let $\overline{\mathbf{S}}_u$ be a local context for $u$. Suppose that there is a neighborhood $U\subseteq \Theta_u$ of $\theta_u$ and a measurable function $Q_u^{\theta_{-u}} : \underline{\mathbf{A}_u} \times \underline{\mathbf{B}_u} \times \underline{\mathbf{S}_u} \to \mathbb{R}$ such that, for every $\eta\in U$, the random variable $J$ is integrable under $\nu_D^{(\eta,\theta_{-u})}$ and
\[
Q_u^{\theta_{-u}}
(\overline{\mathbf{A}}_u,\overline{\mathbf{B}}_u,\overline{\mathbf{S}}_u)
=
\mathbb{E}_{\nu_D^{(\eta,\theta_{-u})}}
[
J
\mid
\sigma(\overline{\mathbf{A}}_u,\overline{\mathbf{B}}_u,\overline{\mathbf{S}}_u)
].
\]
Suppose also that there is a local parameter-gradient representing function $G_{u,Q_u} : \Theta_u \times \underline{\mathbf{A}_u} \times \underline{\mathbf{B}_u} \times \underline{\mathbf{S}_u} \to \Theta_u^*$ for $k_{u,\theta_u}$ relative to $\mathbf{S}_u$ and $Q_u^{\theta_{-u}}$. $G_{u,Q_u}$ denotes the representing function associated with $Q_u^{\theta_{-u}}$. Define
\[
F_u(\eta)
=
\int_{\underline{\mathbf{A}_u}\times \underline{\mathbf{S}_u}}
\left(
\int_{\underline{\mathbf{B}_u}}
Q_u^{\theta_{-u}}(\mathbf{a},\mathbf{b},\mathbf{s})\,
k_{u,\eta}(d\mathbf{b}\mid \mathbf{a})
\right)
\lambda_u^{\theta_{-u}}(d\mathbf{a},d\mathbf{s}).
\]
Assume that $F_u$ is finite in a neighborhood of $\theta_u$ and differentiable at $\theta_u$, and that $G_{u,Q_u} (\theta_u,\overline{\mathbf{A}}_u,\overline{\mathbf{B}}_u,\overline{\mathbf{S}}_u)$ is integrable under $\nu_D^\theta$. Then
\[
D_{\theta_u}\mathcal{L}_D(\theta)
=
\mathbb{E}_{\nu_D^\theta}
[
G_{u,Q_u}
(\theta_u,\overline{\mathbf{A}}_u,\overline{\mathbf{B}}_u,\overline{\mathbf{S}}_u)
].
\]
\end{proposition}

\begin{proof}
For $\eta\in U$, write $\theta^\eta=(\eta,\theta_{-u})$. By the defining property of $Q_u^{\theta_{-u}}$ and the tower property of conditional expectation,
\[
\mathcal{L}_D(\theta^\eta)
=
\mathbb{E}_{\nu_D^{\theta^\eta}}[J]
=
\mathbb{E}_{\nu_D^{\theta^\eta}}
[
Q_u^{\theta_{-u}}
(\overline{\mathbf{A}}_u,\overline{\mathbf{B}}_u,\overline{\mathbf{S}}_u)
].
\]
The random variable $Q_u^{\theta_{-u}}(\overline{\mathbf{A}}_u,\overline{\mathbf{B}}_u,\overline{\mathbf{S}}_u)$ is integrable because it is a conditional expectation of the integrable random variable $J$.

Since $\overline{\mathbf{S}}_u$ is a local context, the preceding expectation can be written as
\[
\mathcal{L}_D(\theta^\eta)
=
\int_{\underline{\mathbf{A}_u}\times \underline{\mathbf{S}_u}}
\left(
\int_{\underline{\mathbf{B}_u}}
Q_u^{\theta_{-u}}(\mathbf{a},\mathbf{b},\mathbf{s})\,
k_{u,\eta}(d\mathbf{b}\mid \mathbf{a})
\right)
\lambda_u^{\theta_{-u}}(d\mathbf{a},d\mathbf{s}).
\]
Thus, $\mathcal{L}_D(\theta^\eta)=F_u(\eta)$ for $\eta\in U$. Therefore,
\[
D_{\theta_u}\mathcal{L}_D(\theta)=D_\eta F_u(\eta)\big|_{\eta=\theta_u}.
\]

By the defining property of the local parameter-gradient representing function,
\[
D_\eta F_u(\eta)\big|_{\eta=\theta_u}
=
\int_{\underline{\mathbf{A}_u}\times \underline{\mathbf{S}_u}}
\left(
\int_{\underline{\mathbf{B}_u}}
G_{u,Q_u}(\theta_u,\mathbf{a},\mathbf{b},\mathbf{s})\,
k_{u,\theta_u}(d\mathbf{b}\mid \mathbf{a})
\right)
\lambda_u^{\theta_{-u}}(d\mathbf{a},d\mathbf{s}).
\]
The local-context identity identifies the law of
$(\overline{\mathbf{A}}_u,\overline{\mathbf{B}}_u,\overline{\mathbf{S}}_u)$ with the measure represented by the last iterated integral. Since the integrand is integrable by assumption, this identity extends from bounded measurable functions to the present integrable function by truncation. Hence
\[
D_{\theta_u}\mathcal{L}_D(\theta)
=
\mathbb{E}_{\nu_D^\theta}
[
G_{u,Q_u}
(\theta_u,\overline{\mathbf{A}}_u,\overline{\mathbf{B}}_u,\overline{\mathbf{S}}_u)
].
\]
This proves the result.
\end{proof}

\subsection{Reverse-mode differentiation for finite acyclic CKSC diagrams}
\label{subsec:reverse_mode_finite_cksc}

The local reverse rule of Proposition~\ref{prop:local_reverse_rule} gives the derivative with respect to one vertex parameter. Since the total parameter space of a finite diagram is a finite product, these local derivatives assemble into the derivative with respect to all parameters.

Let $D$ be a parameterized CKSC diagram. Write $V_{\mathrm{par}}(D)\subseteq V(D)$ for the set of parameterized vertices of $D$. Since the vertices outside $V_{\mathrm{par}}(D)$ have zero parameter spaces, dropping the zero-dimensional factors gives linear isomorphisms
$\Theta_D \cong \prod_{u\in V_{\mathrm{par}}(D)}\Theta_u,~\Theta_D^*\cong\prod_{u\in V_{\mathrm{par}}(D)}\Theta_u^*$. We use these isomorphisms to identify the corresponding spaces in Theorem~\ref{thm:reverse_mode_cksc}.

\begin{theorem}[Reverse-mode differentiation for finite acyclic CKSC diagrams]
\label{thm:reverse_mode_cksc}
Let $D$ be a parameterized CKSC diagram with expected objective $\mathcal{L}_D(\theta)=\mathbb{E}_{\nu_D^\theta}[J]$. Fix $\theta\in \Theta_D$. Assume that $\mathcal{L}_D$ is differentiable at $\theta$. For each $u\in V_{\mathrm{par}}(D)$, assume that the hypotheses of Proposition~\ref{prop:local_reverse_rule} hold at $\theta$ for a local context $\overline{\mathbf{S}}_u$, a conditional expected future objective $Q_u$, and a local parameter-gradient representing function $G_{u,Q_u}$. Define the random covector
\[
\mathcal{G}_u^\theta
=
G_{u,Q_u}
(\theta_u,\overline{\mathbf{A}}_u,\overline{\mathbf{B}}_u,\overline{\mathbf{S}}_u)
\in \Theta_u^*.
\]
Then the derivative of $\mathcal{L}_D$ at $\theta$ is the element of $\Theta_D^*$ whose $u$-component is
$\mathbb{E}_{\nu_D^\theta}[\mathcal{G}_u^\theta]$. Equivalently,
\[
D\mathcal{L}_D(\theta)
=
\left(
\mathbb{E}_{\nu_D^\theta}[\mathcal{G}_u^\theta]
\right)_{u\in V_{\mathrm{par}}(D)}
\in
\prod_{u\in V_{\mathrm{par}}(D)}\Theta_u^*.
\]
\end{theorem}

\begin{proof}
Since $\Theta_D$ is the finite product $\prod_{u\in V_{\mathrm{par}}(D)}\Theta_u$, a tangent vector $\eta\in \Theta_D$ is a tuple $\eta=(\eta_u)_{u\in V_{\mathrm{par}}(D)}$ with $\eta_u\in \Theta_u$. Since $\mathcal{L}_D$ is differentiable at $\theta$, its derivative is determined by its restrictions to the coordinate directions.

Fix $u\in V_{\mathrm{par}}(D)$, and vary only the parameter $\theta_u$, keeping all other vertex parameters fixed. By Proposition~\ref{prop:local_reverse_rule},
\[
D_{\theta_u}\mathcal{L}_D(\theta)
=
\mathbb{E}_{\nu_D^\theta}
[
G_{u,Q_u}
(\theta_u,\overline{\mathbf{A}}_u,\overline{\mathbf{B}}_u,\overline{\mathbf{S}}_u)
]
=
\mathbb{E}_{\nu_D^\theta}[\mathcal{G}_u^\theta].
\]
Thus, the $u$-component of the derivative is $\mathbb{E}_{\nu_D^\theta}[\mathcal{G}_u^\theta]$.

Let $\eta=(\eta_u)_{u\in V_{\mathrm{par}}(D)}\in \Theta_D$. By linearity of the derivative on the finite product,
\[
D\mathcal{L}_D(\theta)[\eta]
=
\sum_{u\in V_{\mathrm{par}}(D)}
D_{\theta_u}\mathcal{L}_D(\theta)[\eta_u].
\]
Substituting the component formula gives
\[
D\mathcal{L}_D(\theta)[\eta]
=
\sum_{u\in V_{\mathrm{par}}(D)}
\mathbb{E}_{\nu_D^\theta}[\mathcal{G}_u^\theta][\eta_u].
\]
This is the action of the covector
$(\mathbb{E}_{\nu_D^\theta}[\mathcal{G}_u^\theta])_{u\in V_{\mathrm{par}}(D)}
\in \prod_{u\in V_{\mathrm{par}}(D)}\Theta_u^* \cong \Theta_D^*$
on $\eta$. Hence, the claimed formula for $D\mathcal{L}_D(\theta)$ follows.
\end{proof}

Theorem~\ref{thm:reverse_mode_cksc} is a reverse-mode statement. The scalar objective $J$ is evaluated at the external outputs of the diagram. For each vertex $u$, the conditional expected future objective $Q_u$ summarizes the contribution of the part of the trace following $u$, conditioned on the local variables needed by the chosen context. The local representing function then converts this quantity into a covector on the local parameter space $\Theta_u$.

Theorem~\ref{thm:reverse_mode_cksc} states that, given a local context, a $\theta_u$-free conditional expected future objective, and a local parameter-gradient representing function at every parameterized vertex, the gradient of the expected objective decomposes into local expectations. The theorem does not construct these quantities or guarantee their existence in arbitrary parameterized CKSC diagrams. Constructing them in concrete classes of diagrams is the role of Propositions~\ref{prop:score_function_local_operator} and~\ref{prop:deterministic_pathwise_local_operator}, together with a diagram-specific choice of context.

\subsection{Gradient transport along CCMP parameter pushforwards}
\label{subsec:gradient_transport_ccmp}

Theorem~\ref{thm:reverse_mode_cksc} applies within a fixed state CMP. A CCMP also has parameter pushforwards between states. Gradients across these pushforwards are transported by the chain rule for differentiable maps.

Let $(F_{\PC},F_\theta)$ be a CCMP over $\T$. Let $\alpha:t\to t'$ be a morphism in $\T$. The corresponding parameter pushforward is $\alpha_{!,\theta}:\Theta_t\to \Theta_{t'}$. Suppose that an expected objective at state $t'$ is given by a differentiable function $\mathcal{L}_{t'}:\Theta_{t'}\to \mathbb{R}$. For example, $\mathcal{L}_{t'}$ may be the expected objective of a parameterized CKSC diagram in $\PC_{t'}$.

\begin{proposition}[Gradient transport along parameter pushforwards]
\label{prop:gradient_transport_parameter_pushforwards}
Let $\alpha:t\to t'$ be a morphism in the indexing category of a CCMP. Assume that $\alpha_{!,\theta}:\Theta_t\to \Theta_{t'}$ is differentiable, and that $\mathcal{L}_{t'}:\Theta_{t'}\to \mathbb{R}$ is differentiable at $\theta_{t'}=\alpha_{!,\theta}(\theta_t)$. Define the pulled-back objective $\mathcal{L}_t^\alpha:\Theta_t\to \mathbb{R}$ by $\mathcal{L}_t^\alpha(\theta_t)=\mathcal{L}_{t'}(\alpha_{!,\theta}(\theta_t))$. Then $\mathcal{L}_t^\alpha$ is differentiable at $\theta_t$, and
\[
D\mathcal{L}_t^\alpha(\theta_t)
=
D\alpha_{!,\theta}(\theta_t)^*
\left(
D\mathcal{L}_{t'}(\theta_{t'})
\right),
\]
where $D\alpha_{!,\theta}(\theta_t)^*:\Theta_{t'}^*\to \Theta_t^*$ is the dual map of the derivative $D\alpha_{!,\theta}(\theta_t):\Theta_t\to \Theta_{t'}$.
\end{proposition}

\begin{proof}
The pulled-back objective is the composite $\mathcal{L}_t^\alpha=\mathcal{L}_{t'}\circ \alpha_{!,\theta}$. By the chain rule for differentiable maps between finite-dimensional real vector spaces,
\[
D\mathcal{L}_t^\alpha(\theta_t)
=
D\mathcal{L}_{t'}(\alpha_{!,\theta}(\theta_t))
\circ
D\alpha_{!,\theta}(\theta_t).
\]
Since $\theta_{t'}=\alpha_{!,\theta}(\theta_t)$, this is
\[
D\mathcal{L}_t^\alpha(\theta_t)
=
D\mathcal{L}_{t'}(\theta_{t'})
\circ
D\alpha_{!,\theta}(\theta_t).
\]
Viewing derivatives of real-valued functions as covectors, composition with $D\alpha_{!,\theta}(\theta_t)$ is the pullback by the dual map $D\alpha_{!,\theta}(\theta_t)^*$. Hence,
\[
D\mathcal{L}_t^\alpha(\theta_t)
=
D\alpha_{!,\theta}(\theta_t)^*
\left(
D\mathcal{L}_{t'}(\theta_{t'})
\right).
\]
\end{proof}

If $\mathcal{L}_{t'}$ is the expected objective of a parameterized CKSC diagram in $\PC_{t'}$, then Theorem~\ref{thm:reverse_mode_cksc} computes $D\mathcal{L}_{t'}(\theta_{t'})$ within the state CMP $\PC_{t'}$. Proposition~\ref{prop:gradient_transport_parameter_pushforwards} then transports this covector back along the CCMP parameter pushforward. Thus, diagrammatic differentiation is fiber-wise, while the CCMP structure supplies the parameter transport between fibers.

\section{Discussion and future work}
\label{sec:discussion}

The main contribution of this paper is the typed polycategorical kernel semantics and its finite acyclic differentiation theory. The results in this paper are stated for finite acyclic diagrams over standard Borel spaces, with fixed diagram shape and fixed interface color morphisms. These hypotheses are part of the formal scope of the construction. Finiteness and acyclicity ensure that trace measures are obtained by iterated kernels and that internal wires can be marginalized without fixed-point semantics. The standard Borel assumption ensures that the conditional expectations used in Section~\ref{sec:diagrammatic_differentiation} can be represented by measurable functions of the conditioning variables. Fixed interface color morphisms enable typed connections under composition. The differentiation theorem is conditional, that is, it identifies the gradient of an expected objective with local expectations under hypotheses on local contexts, conditional expected future objectives, and local gradient representing functions.

The paper does not give a structural treatment of copying, discarding, feedback, or learned interface kernels. Copying, discarding, and learned interface transformations can be represented only when supplied as morphisms or parameterized vertices, while feedback would require additional cyclic or traced structure.

Bayesian extensions are a natural next step. Bayesian parameter learning would require priors on parameter objects, likelihood kernels, and posterior updates, and would therefore interact with categorical accounts of disintegration and Bayesian inversion~\citep{chojacobs2019} and with the treatment of conditioning in Markov categories~\citep{fritz2020}. Bayesian decision theory would require additional objects and morphisms for beliefs, actions, utilities, policies, and value functions. A related compositional direction is given by Bayesian open games, which incorporate stochastic environments, stochastic choices, and incomplete information into compositional game theory~\citep{bolt2023}.

Several categorical directions remain open. One direction is to relate CMPs more directly to Markov categories by adding structural morphisms for copying and discarding. Another direction is to extend the trace semantics to feedback and cyclic diagrams, possibly using traced or guarded categorical structure. Finally, the differentiation result of Section~\ref{sec:diagrammatic_differentiation} can be compared in the future with categorical accounts of gradient-based learning based on lenses, parametrized maps, and reverse derivative categories~\citep{cruttwell2022}. Such comparisons may clarify which parts of reverse-mode differentiation are specific to stochastic kernel diagrams and which are instances of a more general categorical mechanism.

\bibliographystyle{abbrvnat}
\bibliography{references}

@article{fritz2020,
  author  = {Fritz, Tobias},
  title   = {A synthetic approach to {M}arkov kernels, conditional independence and theorems on sufficient statistics},
  journal = {Advances in Mathematics},
  volume  = {370},
  pages   = {107239},
  year    = {2020},
  doi     = {10.1016/j.aim.2020.107239}
}

@book{kallenberg2021,
  author    = {Kallenberg, Olav},
  title     = {Foundations of Modern Probability},
  edition   = {3},
  series    = {Probability Theory and Stochastic Modelling},
  volume    = {99},
  publisher = {Springer},
  address   = {Cham},
  year      = {2021},
  doi       = {10.1007/978-3-030-61871-1}
}

@article{szabo1975,
  author  = {Szabo, Michael E.},
  title   = {Polycategories},
  journal = {Communications in Algebra},
  volume  = {3},
  number  = {8},
  pages   = {663--689},
  year    = {1975},
  doi     = {10.1080/00927877508822067}
}

@article{garner2008,
  author  = {Garner, Richard},
  title   = {Polycategories via pseudo-distributive laws},
  journal = {Advances in Mathematics},
  volume  = {218},
  number  = {3},
  pages   = {781--827},
  year    = {2008},
  doi     = {10.1016/j.aim.2008.02.001}
}

@article{cockettseely1997,
  author  = {Cockett, J. R. B. and Seely, R. A. G.},
  title   = {Weakly distributive categories},
  journal = {Journal of Pure and Applied Algebra},
  volume  = {114},
  number  = {2},
  pages   = {133--173},
  year    = {1997},
  doi     = {10.1016/0022-4049(95)00160-3}
}

@article{koslowski2005,
  author  = {Koslowski, J{\"u}rgen},
  title   = {A monadic approach to polycategories},
  journal = {Theory and Applications of Categories},
  volume  = {14},
  number  = {7},
  pages   = {125--156},
  year    = {2005}
}

@article{chojacobs2019,
  author  = {Cho, Kenta and Jacobs, Bart},
  title   = {Disintegration and {B}ayesian inversion via string diagrams},
  journal = {Mathematical Structures in Computer Science},
  volume  = {29},
  number  = {7},
  pages   = {938--971},
  year    = {2019},
  doi     = {10.1017/S0960129518000488}
}

@article{bolt2023,
  author  = {Bolt, Joe and Hedges, Jules and Zahn, Philipp},
  title   = {Bayesian open games},
  journal = {Compositionality},
  volume  = {5},
  number  = {9},
  year    = {2023},
  doi     = {10.32408/compositionality-5-9}
}

@inproceedings{cruttwell2022,
  author    = {Cruttwell, G. S. H. and Gavranovi{\'c}, Bruno and Ghani, Neil and Wilson, Paul and Zanasi, Fabio},
  title     = {Categorical Foundations of Gradient-Based Learning},
  booktitle = {Programming Languages and Systems},
  series    = {Lecture Notes in Computer Science},
  volume    = {13240},
  pages     = {1--28},
  publisher = {Springer},
  year      = {2022},
  doi       = {10.1007/978-3-030-99336-8_1}
}

@book{johnstone2002,
  author    = {Johnstone, Peter T.},
  title     = {Sketches of an Elephant: A Topos Theory Compendium},
  volume    = {1},
  series    = {Oxford Logic Guides},
  number    = {43},
  publisher = {Clarendon Press},
  address   = {Oxford},
  year      = {2002}
}

@book{jacobs1999,
  author    = {Jacobs, Bart},
  title     = {Categorical Logic and Type Theory},
  series    = {Studies in Logic and the Foundations of Mathematics},
  volume    = {141},
  publisher = {Elsevier},
  address   = {Amsterdam},
  year      = {1999}
}

\appendix

\section{KSC proofs (of Section~\ref{sec:markov_polycategories_ksc})}
\label{app:ksc_proofs}

\begin{proof}[Proof of Theorem~\ref{thm:trace_semantics_ksc}]
Fix a topological ordering $u_1,\ldots,u_N$ of $V(D)$. For each external input value $\mathbf{x}\in \underline{\mathbf{I}_D}$, the iterated product in Definition~\ref{def:trace_measure_ksc} defines a probability measure on $\prod_{\ell=1}^N\underline{\mathbf{B}_{u_\ell}}$. This follows by finite iteration of the defining property of Markov kernels. At each step $\ell$, the input map $R_\ell$ is measurable because it is built from coordinate projections and tuple formation.

We show that the trace kernel is measurable in the external input. Let $E\in \Sigma_{\underline{\mathbf{O}_D}}$. The map $P_{\mathrm{out}}$ is measurable, so the function
\[
(\mathbf{y}_1,\ldots,\mathbf{y}_N)
\mapsto
\mathbf{1}_E(P_{\mathrm{out}}(\mathbf{y}_1,\ldots,\mathbf{y}_N))
\]
is bounded and measurable. By induction on $\ell$, if $\varphi$ is any bounded measurable function on $\prod_{s=1}^{\ell}\underline{\mathbf{B}_{u_s}}$, then the function
\[
\mathbf{x}
\mapsto
\int
\varphi(\mathbf{y}_1,\ldots,\mathbf{y}_\ell)
\prod_{s=1}^{\ell}
k_{u_s}
\bigl(
d\mathbf{y}_s
\mid
R_s(\mathbf{x},\mathbf{y}_1,\ldots,\mathbf{y}_{s-1})
\bigr)
\]
is measurable. The induction step uses the fact that integration of a bounded measurable function against a Markov kernel gives a measurable function of the conditioning variable. Taking $\ell=N$ and the displayed indicator function proves that $\mathbf{x}\mapsto K_D^{u_1,\ldots,u_N}(E\mid \mathbf{x})$ is measurable. For fixed $\mathbf{x}$, the map $E\mapsto K_D^{u_1,\ldots,u_N}(E\mid \mathbf{x})$ is the pushforward of a probability measure along the measurable map $P_{\mathrm{out}}$, hence is a probability measure. Thus, $K_D^{u_1,\ldots,u_N}$ is a Markov kernel.

It remains to prove independence of the topological ordering. We identify products indexed by different topological orderings by the labels of their vertex coordinates. Thus all measures are regarded as measures on the same labelled product space $\prod_{u\in V(D)}\underline{\mathbf{B}_u}$.

Any two topological orderings of a finite acyclic directed graph are connected by a finite sequence of adjacent swaps of incomparable vertices. It is therefore enough to prove invariance under one such swap. Let two topological orderings differ only by interchanging adjacent vertices $a$ and $b$. Since both orderings are topological, there is no directed path from $a$ to $b$ or from $b$ to $a$ in the part of the ordering being swapped. In particular, the input map for $a$ does not depend on the output of $b$, and the input map for $b$ does not depend on the output of $a$.

Let $\Phi$ be a bounded measurable function on $\prod_{u\in V(D)}\underline{\mathbf{B}_u}$. After all kernels before $a$ and $b$ have been integrated, and before the kernels after $a$ and $b$ are integrated, the two relevant integrations have the form
\[
\int
\int
G(\eta,\mathbf{y}_a,\mathbf{y}_b)
\,k_a(d\mathbf{y}_a\mid R_a(\eta))
\,k_b(d\mathbf{y}_b\mid R_b(\eta)),
\]
or the same expression with the two kernel integrations interchanged. Here, $\eta$ denotes the already sampled earlier output values, and $G$ is the bounded measurable function obtained by integrating $\Phi$ against all later kernels. The function $G$ is measurable by the same kernel-measurability argument used above. Since both integrations are against probability measures, the Fubini-Tonelli theorem gives equality of the two iterated integrals. Hence, swapping adjacent incomparable vertices does not change the integral of $\Phi$.

Applying this to $\Phi=\mathbf{1}_E\circ P_{\mathrm{out}}$ shows that the two adjacent topological orderings give the same trace kernel. Since any two topological orderings are connected by adjacent swaps of incomparable vertices, the trace kernel is independent of the chosen topological ordering.
\end{proof}

\begin{proof}[Proof of Corollary~\ref{cor:ksc_unitality}]
We prove equality of Markov kernels.

First fix an input slot $j\in \{1,\ldots,m\}$. The composite $h\circ_{(1,j)}\id_{A_j}$ connects the unique output of $\id_{A_j}$ to the $j$-th input of $h$. Its input profile is
\[
\mathbf{A}=(A_1,\ldots,A_{j-1},A_j,A_{j+1},\ldots,A_m),
\]
and its output profile is $\mathbf{B}$.

Let $\mathbf{a}=(a_1,\ldots,a_m)\in \underline{\mathbf{A}}$ and let $E\in \Sigma_{\underline{\mathbf{B}}}$. By Definition~\ref{def:ksc},
\[
(h\circ_{(1,j)}\id_{A_j})(E\mid \mathbf{a})
=
\int_{\underline{A_j}}
\left(
\int_{\underline{\mathbf{B}}}
\mathbf{1}_E(\mathbf{b})\,
h(d\mathbf{b}\mid (a_1,\ldots,a_{j-1},z,a_{j+1},\ldots,a_m))
\right)
\delta_{a_j}(dz).
\]
The outer integral evaluates the integrand at $z=a_j$. Hence
\[
(h\circ_{(1,j)}\id_{A_j})(E\mid \mathbf{a})
=
\int_{\underline{\mathbf{B}}}
\mathbf{1}_E(\mathbf{b})\,h(d\mathbf{b}\mid \mathbf{a})
=
h(E\mid \mathbf{a}).
\]
Thus $h\circ_{(1,j)}\id_{A_j}=h$.

Now fix an output slot $i\in \{1,\ldots,n\}$. The composite $\id_{B_i}\circ_{(i,1)}h$ connects the $i$-th output of $h$ to the unique input of $\id_{B_i}$. Its input profile is $\mathbf{A}$, and its output profile is
\[
\mathbf{B}=(B_1,\ldots,B_{i-1},B_i,B_{i+1},\ldots,B_n).
\]
Let $\mathbf{a}\in \underline{\mathbf{A}}$ and let $E\in \Sigma_{\underline{\mathbf{B}}}$. By Definition~\ref{def:ksc},
\[
(\id_{B_i}\circ_{(i,1)}h)(E\mid \mathbf{a})
=
\int_{\underline{\mathbf{B}}}
\left(
\int_{\underline{B_i}}
\mathbf{1}_E(b_1,\ldots,b_{i-1},z,b_{i+1},\ldots,b_n)\,
\delta_{b_i}(dz)
\right)
h(d\mathbf{b}\mid \mathbf{a}).
\]
The inner integral evaluates the integrand at $z=b_i$. Therefore
\[
(\id_{B_i}\circ_{(i,1)}h)(E\mid \mathbf{a})
=
\int_{\underline{\mathbf{B}}}
\mathbf{1}_E(\mathbf{b})\,h(d\mathbf{b}\mid \mathbf{a})
=
h(E\mid \mathbf{a}).
\]
Thus $\id_{B_i}\circ_{(i,1)}h=h$.
\end{proof}

\begin{proof}[Proof of Proposition~\ref{prop:ksc_compatible_trace_semantics}]
Choose topological orderings of $D_1$ and $D_2$. Since the new wire is directed from $D_1$ to $D_2$, concatenating these two topological orderings gives a topological ordering of $D$. By Theorem~\ref{thm:trace_semantics_ksc}, the trace kernel of $D$ is independent of this choice.

Let the external input of $D$ be written as $V_j(\mathbf{x}_1,\mathbf{x}_{2,-j})$, where $\mathbf{x}_1$ is an input value for $D_1$ and $\mathbf{x}_{2,-j}$ gives the external inputs of $D_2$ other than the connected slot. Let $\mathbf{y}_1$ denote the external output value of $D_1$, and let $\mathbf{y}_{1,-i}$ be this output value with its $i$-th component removed. Let $z=\pi_i(\mathbf{y}_1)$ be the value carried by the connected wire.

By definition of trace semantics, integrating over the internal wires of $D_1$ gives the kernel $K_{D_1}(d\mathbf{y}_1\mid \mathbf{x}_1)$. Conditional on $\mathbf{y}_1$, the input to $D_2$ is $T_j(\mathbf{x}_{2,-j},z)$, and integrating over the internal wires of $D_2$ gives the kernel $K_{D_2}(d\mathbf{y}_2\mid T_j(\mathbf{x}_{2,-j},\pi_i(\mathbf{y}_1)))$. The external output of $D$ is $U_i(\pi_{-i}(\mathbf{y}_1),\mathbf{y}_2)$.

Therefore, for every measurable set $E$ in the external output space of $D$,
\[
K_D(E\mid V_j(\mathbf{x}_1,\mathbf{x}_{2,-j}))
=
\int
\left(
\int
\mathbf{1}_E
\left(
U_i(\pi_{-i}(\mathbf{y}_1),\mathbf{y}_2)
\right)
\,
K_{D_2}
(d\mathbf{y}_2\mid T_j(\mathbf{x}_{2,-j},\pi_i(\mathbf{y}_1)))
\right)
K_{D_1}(d\mathbf{y}_1\mid \mathbf{x}_1).
\]
This is the KSC formula for $K_{D_2}\circ_{(i,j)}K_{D_1}$. Hence, $K_D=K_{D_2}\circ_{(i,j)}K_{D_1}$.
\end{proof}

\section{CKSC proofs (of Section~\ref{sec:cksc_diagram_semantics})}
\label{app:cksc_proofs}

\begin{proof}[Proof of Corollary~\ref{cor:cksc_unitality}]
We prove equality of Markov kernels.

First fix an input slot $j\in \{1,\ldots,m\}$. By Definition~\ref{def:interface_system}, the identity color morphism $\id_{\chi(A_j)}$ is an interface color morphism from $A_j$ to $A_j$, and
\[
\kappa_{\id_{\chi(A_j)};A_j,A_j}=\id_{A_j}.
\]
By Definition~\ref{def:cksc},
\[
h\circ^{\id_{\chi(A_j)}}_{(1,j)}\id_{A_j}
=
h\circ_{(1,j)}
\bigl(
\kappa_{\id_{\chi(A_j)};A_j,A_j}\circ_{(1,1)}\id_{A_j}
\bigr).
\]
Substituting the identity interface kernel gives
\[
h\circ^{\id_{\chi(A_j)}}_{(1,j)}\id_{A_j}
=
h\circ_{(1,j)}
(\id_{A_j}\circ_{(1,1)}\id_{A_j}).
\]
By Corollary~\ref{cor:ksc_unitality}, $\id_{A_j}\circ_{(1,1)}\id_{A_j}=\id_{A_j}$ and $h\circ_{(1,j)}\id_{A_j}=h$. Hence,
\[
h\circ^{\id_{\chi(A_j)}}_{(1,j)}\id_{A_j}=h.
\]

Now fix an output slot $i\in \{1,\ldots,n\}$. By Definition~\ref{def:interface_system},
\[
\kappa_{\id_{\chi(B_i)};B_i,B_i}=\id_{B_i}.
\]
Using Definition~\ref{def:cksc}, we obtain
\[
\id_{B_i}\circ^{\id_{\chi(B_i)}}_{(i,1)}h
=
\id_{B_i}\circ_{(i,1)}
\bigl(
\kappa_{\id_{\chi(B_i)};B_i,B_i}\circ_{(i,1)}h
\bigr).
\]
Substituting the identity interface kernel gives
\[
\id_{B_i}\circ^{\id_{\chi(B_i)}}_{(i,1)}h
=
\id_{B_i}\circ_{(i,1)}
(\id_{B_i}\circ_{(i,1)}h).
\]
By Corollary~\ref{cor:ksc_unitality}, $\id_{B_i}\circ_{(i,1)}h=h$, and composing again with $\id_{B_i}$ along the same output slot gives $h$. Therefore,
\[
\id_{B_i}\circ^{\id_{\chi(B_i)}}_{(i,1)}h=h.
\]
\end{proof}

\begin{proof}[Proof of Proposition~\ref{prop:cksc_compatible_trace_semantics}]
Let $\widehat{D}_1$ and $\widehat{D}_2$ be the interface expansions of $D_1$ and $D_2$, respectively. The interface expansion $\widehat{D}$ of $D$ is obtained from $\widehat{D}_1$ and $\widehat{D}_2$ by inserting the unary interface vertex labelled by $\kappa_{f;B,C}:B\to C$ between the selected external output slot of $\widehat{D}_1$ and the selected external input slot of $\widehat{D}_2$.

By Definition~\ref{def:trace_kernel_cksc}, $K_{D_1}=K_{\widehat{D}_1}$, $K_{D_2}=K_{\widehat{D}_2}$, and $K_D=K_{\widehat{D}}$. By Proposition~\ref{prop:ksc_compatible_trace_semantics}, connecting $\widehat{D}_1$ to the unary interface vertex gives the KSC trace kernel
\[
\kappa_{f;B,C}\circ_{(i,1)}K_{\widehat{D}_1}.
\]
Applying Proposition~\ref{prop:ksc_compatible_trace_semantics} once more to connect this composite to $\widehat{D}_2$ along the $j$-th input slot gives
\[
K_{\widehat{D}}
=
K_{\widehat{D}_2}
\circ_{(i,j)}
\left(
\kappa_{f;B,C}\circ_{(i,1)}K_{\widehat{D}_1}
\right).
\]
Substituting $K_{D_1}=K_{\widehat{D}_1}$, $K_{D_2}=K_{\widehat{D}_2}$, and $K_D=K_{\widehat{D}}$ gives
\[
K_D
=
K_{D_2}\circ^f_{(i,j)}K_{D_1}
\]
by the CKSC formula of Definition~\ref{def:cksc}.
\end{proof}

\section{Functorial proofs (of Section~\ref{sec:ccmp})}
\label{app:functorial_proofs}

\begin{proof}[Proof of Proposition~\ref{prop:category_of_cmps}]
Let $\PC$ be a CMP over $(\K,\M,\iota)$. The identity map on objects and morphisms defines a CMP-functor $\Id_{\PC}:\PC\to \PC$. It preserves identity kernels and KSC because these structures are unchanged. It preserves object colors, morphism colors, and interface kernels by definition.

Now let $G:\PC\to \PC',~H:\PC'\to \PC''$ be CMP-functors over the color system $(\K,\M,\iota)$. Define $H G:\PC\to \PC''$ by composition of the object and morphism assignments: $(HG)(X)=H(GX),~(HG)(k)=H(Gk)$. For every object $X$,
\[
(HG)(\id_X)=H(G(\id_X))=H(\id_{GX})=\id_{H G X}.
\]
If $l\circ_{(i,j)}k$ is defined in $\PC$, then
\[
(HG)(l\circ_{(i,j)}k)
=
H(G(l\circ_{(i,j)}k))
=
H(Gl\circ_{(i,j)}Gk)
=
HGl\circ_{(i,j)}HGk.
\]
Thus, $HG$ preserves identities and KSC.

The composite $HG$ preserves object colors because $\chi''(HGX)=\chi'(GX)=\chi(X)$. It preserves morphism colors because $\psi''(HGk)=\psi'(Gk)=\psi(k)$.

It remains to check preservation of interface systems. Let $B,C\in \Ob(\PC)$ and $f\in \Int_{\PC}(B,C)$. Since $G$ preserves interfaces, $f\in \Int_{\PC'}(GB,GC)$ and $G(\kappa_{f;B,C})=\kappa'_{f;GB,GC}$. Since $H$ preserves interfaces, $f\in \Int_{\PC''}(HGB,HGC)$ and $H(\kappa'_{f;GB,GC})=\kappa''_{f;HGB,HGC}$. Therefore,
\[
(HG)(\kappa_{f;B,C})
=
H(G(\kappa_{f;B,C}))
=
H(\kappa'_{f;GB,GC})
=
\kappa''_{f;HGB,HGC}.
\]
Hence, $HG$ is a CMP-functor.

Composition of CMP-functors is associative because it is composition of the underlying object and morphism assignments. The identity CMP-functors are units for the same reason. Thus, CMPs over $(\K,\M,\iota)$ and CMP-functors form a category.
\end{proof}

\begin{proof}[Proof of Proposition~\ref{prop:cmp_functors_preserve_cksc_reductions}]
Choose an iterated CKSC expression $E$ realizing $D$. Applying $G$ to the vertex kernels in $E$, while leaving the interface color morphisms unchanged, gives an iterated CKSC expression over $GD$. The same interface color morphisms label the internal wires of $GD$, which is well-defined because $G$ preserves interface systems.

Each evaluation kernel of an iterated CKSC expression is a morphism of $\PC$, because it is obtained from vertex kernels and interface kernels by finitely many KSC operations, and the underlying Markov polycategory of $\PC$ is closed under KSC. Hence, $G$ may be applied to evaluation kernels.

We prove by induction on the construction of $E$ that the evaluation kernel of the image expression is the image under $G$ of the evaluation kernel of $E$.

If $E$ is a single vertex labelled by $k$, then the claim is immediate. The image expression is the single vertex labelled by $Gk$, and its evaluation kernel is $Gk$.

For the induction step, suppose that $E$ is formed from $(E_1,E_2,e)$, where $e$ carries the interface color morphism $f_e$. Let $H_1$ and $H_2$ be the evaluation kernels of $E_1$ and $E_2$, respectively. Suppose that the wire $e$ connects the $r$-th external output slot of the subdiagram of $E_1$ to the $s$-th external input slot of the subdiagram of $E_2$. The evaluation kernel of $E$ is $H_2\circ^{f_e}_{(r,s)}H_1$. By Definition~\ref{def:cksc},
\[
H_2\circ^{f_e}_{(r,s)}H_1
=
H_2\circ_{(r,s)}
\left(
\kappa_{f_e;B,C}\circ_{(r,1)}H_1
\right),
\]
where $B$ and $C$ are the connected output and input objects.

Since $G$ preserves KSC and interface kernels,
\[
G\bigl(H_2\circ^{f_e}_{(r,s)}H_1\bigr)
=
G H_2
\circ_{(r,s)}
\left(
\kappa'_{f_e;GB,GC}\circ_{(r,1)}G H_1
\right).
\]
By Definition~\ref{def:cksc}, the right-hand side is $G H_2\circ^{f_e}_{(r,s)}G H_1$. By the induction hypothesis, $G H_1$ and $G H_2$ are the evaluation kernels of the image expressions corresponding to $E_1$ and $E_2$. Hence, the image expression corresponding to $E$ has evaluation kernel $G(H_2\circ^{f_e}_{(r,s)}H_1)$.

Thus, the evaluation kernel of the image expression is the image under $G$ of the evaluation kernel of $E$. Since $E$ realizes $D$, the image expression realizes $GD$. Therefore, $GD$ is CKSC-reducible. By Corollary~\ref{cor:cksc_associativity_interchange}, the evaluation kernel of $E$ is $\langle D\rangle_{\PC}$, and the evaluation kernel of the image expression is $\langle GD\rangle_{\PC'}$. Hence, $G\langle D\rangle_{\PC} = \langle GD\rangle_{\PC'}$.

Finally, by Corollary~\ref{cor:cksc_associativity_interchange}, the morphism $\langle GD\rangle_{\PC'}$ has underlying Markov kernel equal to the trace kernel of $GD$.
\end{proof}

\section{Differentiation proofs (of Section~\ref{sec:diagrammatic_differentiation})}
\label{app:differentiation_proofs}

\begin{proof}[Proof of Proposition~\ref{prop:score_function_local_operator}]
Let $\lambda$ be a probability measure on $\underline{\mathbf{A}}\times \underline{\mathbf{S}}$ which does not depend on $\theta$, and let $\lambda_{\mathbf{A}}$ be its marginal on $\underline{\mathbf{A}}$. Define
\[
F_Q(\theta)
=
\int_{\underline{\mathbf{A}}\times \underline{\mathbf{S}}}
\left(
\int_{\underline{\mathbf{B}}}
Q(\mathbf{a},\mathbf{b},\mathbf{s})
p_\theta(\mathbf{b}\mid \mathbf{a})
\,\nu_{\mathbf{B}}(d\mathbf{b})
\right)
\lambda(d\mathbf{a},d\mathbf{s}).
\]
By the assumed interchange of differentiation and integration,
\[
D_\theta F_Q(\theta)
=
\int
Q(\mathbf{a},\mathbf{b},\mathbf{s})
D_\theta p_\theta(\mathbf{b}\mid \mathbf{a})
\,\nu_{\mathbf{B}}(d\mathbf{b})
\,\lambda(d\mathbf{a},d\mathbf{s}).
\]
On the set where $p_\theta(\mathbf{b}\mid \mathbf{a})>0$,
\[
D_\theta p_\theta(\mathbf{b}\mid \mathbf{a})
=
p_\theta(\mathbf{b}\mid \mathbf{a})
D_\theta \log p_\theta(\mathbf{b}\mid \mathbf{a}).
\]
Since $p_\theta(\mathbf{b}\mid \mathbf{a})>0$ for $(\lambda_{\mathbf{A}}\otimes \nu_{\mathbf{B}})$-almost every $(\mathbf{a},\mathbf{b})$, this substitution changes the integrand only on a set whose $(\mathbf{a},\mathbf{b})$-projection is $(\lambda_{\mathbf{A}}\otimes \nu_{\mathbf{B}})$-null, so the value of the integral is unchanged. Hence,
\[
D_\theta F_Q(\theta)
=
\int
Q(\mathbf{a},\mathbf{b},\mathbf{s})
D_\theta \log p_\theta(\mathbf{b}\mid \mathbf{a})
p_\theta(\mathbf{b}\mid \mathbf{a})
\,\nu_{\mathbf{B}}(d\mathbf{b})
\,\lambda(d\mathbf{a},d\mathbf{s}).
\]
The set where $p_\theta(\mathbf{b}\mid \mathbf{a})=0$ carries no mass under $k_\theta(d\mathbf{b}\mid \mathbf{a})=p_\theta(\mathbf{b}\mid \mathbf{a})\,\nu_{\mathbf{B}}(d\mathbf{b})$, so extending $G_{k,Q}^{\mathrm{sf}}$ by zero on this set does not affect integrals against $k_\theta$. Therefore,
\[
D_\theta F_Q(\theta)
=
\int_{\underline{\mathbf{A}}\times \underline{\mathbf{S}}}
\left(
\int_{\underline{\mathbf{B}}}
G_{k,Q}^{\mathrm{sf}}(\theta,\mathbf{a},\mathbf{b},\mathbf{s})
\,k_\theta(d\mathbf{b}\mid \mathbf{a})
\right)
\lambda(d\mathbf{a},d\mathbf{s}).
\]
This is the representing property of Definition~\ref{def:admissible_local_parameter_gradient_operator}.
\end{proof}

\begin{proof}[Proof of Proposition~\ref{prop:deterministic_pathwise_local_operator}]
Let $\lambda$ be a probability measure on $\underline{\mathbf{A}}\times \underline{\mathbf{S}}$ which does not depend on $\theta$. Since $k_\theta(d\mathbf{b}\mid \mathbf{a})=\delta_{g(\theta,\mathbf{a})}(d\mathbf{b})$, the function $F_Q$ of Definition~\ref{def:admissible_local_parameter_gradient_operator} is
\[
F_Q(\theta)
=
\int_{\underline{\mathbf{A}}\times \underline{\mathbf{S}}}
Q(\mathbf{a},g(\theta,\mathbf{a}),\mathbf{s})\,
\lambda(d\mathbf{a},d\mathbf{s}).
\]
By the assumed interchange of differentiation and integration,
\[
D_\theta F_Q(\theta)
=
\int_{\underline{\mathbf{A}}\times \underline{\mathbf{S}}}
D_\theta
\bigl[
Q(\mathbf{a},g(\theta,\mathbf{a}),\mathbf{s})
\bigr]
\lambda(d\mathbf{a},d\mathbf{s}).
\]
For each fixed $(\mathbf{a},\mathbf{s})$, the chain rule gives
\[
D_\theta
\bigl[
Q(\mathbf{a},g(\theta,\mathbf{a}),\mathbf{s})
\bigr]
=
D_{\mathbf{b}}Q(\mathbf{a},g(\theta,\mathbf{a}),\mathbf{s})
\circ
D_\theta g(\theta,\mathbf{a}).
\]
Thus,
\[
D_\theta F_Q(\theta)
=
\int_{\underline{\mathbf{A}}\times \underline{\mathbf{S}}}
G_{k,Q}^{\mathrm{det}}
(\theta,\mathbf{a},g(\theta,\mathbf{a}),\mathbf{s})
\lambda(d\mathbf{a},d\mathbf{s}).
\]
Since integration against $k_\theta(d\mathbf{b}\mid \mathbf{a})$ is integration against $\delta_{g(\theta,\mathbf{a})}$, this is equivalent to
\[
D_\theta F_Q(\theta)
=
\int_{\underline{\mathbf{A}}\times \underline{\mathbf{S}}}
\left(
\int_{\underline{\mathbf{B}}}
G_{k,Q}^{\mathrm{det}}(\theta,\mathbf{a},\mathbf{b},\mathbf{s})
\,k_\theta(d\mathbf{b}\mid \mathbf{a})
\right)
\lambda(d\mathbf{a},d\mathbf{s}).
\]
This is the representing property of Definition~\ref{def:admissible_local_parameter_gradient_operator}.
\end{proof}

\begin{proof}[Proof of Lemma~\ref{lem:local_context_extension}]
Consider first a non-negative measurable function $\varphi$. For $n\in \mathbb{N}$, the truncation $\varphi_n=\min(\varphi,n)$ is bounded and measurable, so the identity in Definition~\ref{def:admissible_local_context} holds for $\varphi_n$. As $n\to\infty$, the monotone convergence theorem applies to both sides: on the left-hand side with respect to the joint law of $(\overline{\mathbf{A}}_u,\overline{\mathbf{B}}_u,\overline{\mathbf{S}}_u)$ under $\nu_D^\theta$, and on the right-hand side first for the inner integral against $k_{u,\theta_u}(\,\cdot\mid \mathbf{a})$ and then for the outer integral against $\lambda_u^{\theta_{-u}}$. Hence, the identity holds for every non-negative measurable $\varphi$, with both sides possibly infinite.

Next, let $\varphi$ be measurable and integrable with respect to the joint law of $(\overline{\mathbf{A}}_u,\overline{\mathbf{B}}_u,\overline{\mathbf{S}}_u)$, and write $\varphi=\varphi^+-\varphi^-$. Applying the nonnegative case to $\varphi^+$ and $\varphi^-$ shows that all four integrals involved are finite, and subtracting the two identities gives the identity for $\varphi$, with both sides finite.

For the $\Theta_u^*$-valued case, fix a basis of $\Theta_u^*$ and write $\Phi=(\Phi_1,\ldots,\Phi_p)$ in coordinates. Each component $\Phi_i$ is integrable by assumption, so the scalar case applies componentwise, and the identity holds coordinatewise. The resulting identity of covectors does not depend on the chosen basis, because both sides are linear in $\Phi$ and a change of basis acts by the same invertible linear map on both sides.
\end{proof}

\end{document}